\documentclass[11pt]{article}
\usepackage{amsmath,amssymb,amsthm}
\usepackage[colorlinks=true]{hyperref}
\def\a{\mathbf{a}}
\def\c{\mathbf{c}}

\def\b{\mathbf{b}}
\def\f{\mathbf{f}}
\def\m{\mathbf{m}}

\def\etta{\boldsymbol{\eta}}

\def\llambda{\boldsymbol{\lambda}}

\def\N{\mathbb{N}}

\def\C{\mathbb{C}}

\newtheorem{theorem}{\hspace*{\parindent}Theorem}
\newtheorem{lemma}{\hspace*{\parindent}Lemma}
\newtheorem{corollary}{\hspace*{\parindent}Corollary}
\newcounter{remark}
\newcommand{\remark}{%
  \refstepcounter{remark}%
  \par\noindent\textbf{Remark \theremark.}\ }

\title{Polynomial perturbations of Euler's and Clausen's identities}
\author{Dmitrii Karp$^{\rm a,b}$\footnote{E-mail: D. Karp -- \emph{dimkrp@gmail.com}}
	\\[10pt]
	\small{\textit{$\phantom{1}^a$Holon Institute of Technology, Holon, Israel}}
    \\[5pt]
	\small{\textit{$\phantom{1}^b$Institute of Mathematics and Informatics, BAS, Sofia, Bulgaria}}
}

\date{{\it Dedicated to the memory of Allen Miller and Richard Paris}}
\begin{document}
	\maketitle
	
	
\begin{abstract}
	A product of two hypergeometric series is generally not hypergeometric. However, there are a few cases when such product does reduce to a single hypergeometric series. The oldest result of this type, beyond the obvious $(1-x)^{a}(1-x)^{b}=(1-x)^{a+b}$, is Euler's transformation for the Gauss hypergeometric function ${}_2F_1$. Another important one is the celebrated Clausen's identity dated 1828  which expresses the square of a suitable ${}_2F_1$ function as a single ${}_3F_2$.  By equating coefficients each product identity corresponds to a special type of summation theorem for terminating series. Over the last two decades Euler's transformations and many summation theorems have been extended by introducing additional parameter pairs differing by positive integers. This amounts to multiplication of the power series coefficients by values of a fixed polynomial at nonnegative integers.  The main goal of this paper is to present an extension of Clausen's identity obtained by such polynomial perturbation.  To this end, we first reconsider the polynomial perturbations of Euler's transformations found by Miller and Paris around 2010. We propose new, simplified proofs of their transformations relating them to polynomial interpolation and exhibiting various new forms of the characteristic polynomials. We further introduce the notion of the Miller-Paris operators which play a prominent role in the construction of the extended Clausen's identity. 
	\end{abstract}
	
	\bigskip
	
	
	\bigskip
	
	
	\bigskip
	
\section{Introduction and preliminaries}
A series $\sum_{k}c_k$ is called hypergeometric if $c_{k+1}/c_{k}$ is a rational function of $k$.  Suppose $R(t)$ is another rational function.  Then, clearly, the series  $\sum_{k}R(k)c_k$ remains hypergeometric and it is natural to call it a {\it rational perturbation} of the original series $\sum_{k}c_k$. In particular, when $R(t)=P(t)$ reduces to a polynomial of some fixed  positive degree $d$, we will call $\sum_{k}P(k)c_k$ a {\it polynomial perturbation} of the original series $\sum_{k}c_k$.  If one knows the zeros of $P$, say $\llambda=(\lambda_1,\ldots,\lambda_d)$, then
$$
P(k)=P(0)\frac{(k-\lambda_1)\cdots(k-\lambda_d)}{(-\lambda_1)\cdots(-\lambda_d)}=P(0)\frac{(1-\lambda_1)_{k}\cdots(1-\lambda_d)_k}{(-\lambda_1)_k\cdots(-\lambda_d)_k},
$$
where the standard notation $(a)_k=\Gamma(a+k)/\Gamma(a)$ for the rising factorial (or Pochhammer's symbol) has been used. We will assume that $P(0)\ne0$ throughout the paper. This entails no loss of generality as if $P(0)=0$ we can start summation from $k=1$ and replace the index of summation $k-1\to k$.  Then, assuming that the original series is written in the standard form
$$
\sum_{k}c_k=\sum_{k=0}^{\infty}\frac{(\a)_{k}}{(\b)_kk!}x^k=
{}_{p}F_{q}\!\left(\left.\begin{matrix}\a\\\b\end{matrix}\:\right\vert x\right),
$$
where from here onward we will use the shorthand notation for the product $(\a)_{k}=(a_1)_{k}\cdots(a_p)_{k}$,
we obtain ($\llambda+\alpha$ is understood element-wise for a scalar $\alpha$)
$$
\sum_{k}P(k)c_k=\sum_{k=0}^{\infty}\frac{(\a)_{k}}{(\b)_kk!}P(k)x^k=
P(0){}_{p+d}F_{q+d}\!\left(\left.\begin{matrix}\a,1-\llambda\\\b,-\llambda\end{matrix}\:\right\vert x\right).  
$$
The function on the right hand side is frequently called hypergeometric function with integral parameter differences following the work of Karsslon \cite{Karlsson}. The acronym IPD (and $q$-IPD for its $q$ extension) was coined by Michael Schlosser \cite{Schlosser} to describe a hypergeometric function containing a group of upper/lower parameter pairs differing by positive integers. The above definition, however, only involves the values of the polynomial $P$ at nonnegative integers and does not assume the knowledge of its (usually unknown) zeros. Hence, the following notation for the rational/polynomial perturbation turns out to be very useful:
$$
\sum_{k=0}^{\infty}\frac{(\a)_{k}}{(\b)_kk!}R(k)x^k=F\!\left(\begin{matrix}\a\\\b\end{matrix}\,\bigg\vert\,R\,\bigg\vert\, x\right), 
$$
so that 
$$
F\!\left(\begin{matrix}\a\\\b\end{matrix}\:\bigg\vert P\:\bigg\vert x\right)=P(0){}_{p+d}F_{q+d}\!\left(\left.\begin{matrix}\a,1-\llambda\\\b,-\llambda\end{matrix}\:\right\vert x\right).
$$
In case of rational perturbation $R=P/Q$ we can also factor the denominator polynomial $Q$ resulting in
$$
F\!\left(\begin{matrix}\a\\\b\end{matrix}\:\bigg\vert R\:\bigg\vert x\right)
=R(0){}_{p+d+r}F_{q+d+r}\!\left(\left.\begin{matrix}\a,1-\llambda,-\etta\\\b,-\llambda,1-\etta\end{matrix}\:\right\vert x\right),
$$
where $\etta=(\eta_1,\ldots,\eta_r)$ are the zeros of $Q$. 

We note that the contiguous hypergeometric series ${}_pF_{q}(\a+\m_1;\b+\m_2;x)$, where $\m_1$, $\m_2$ are integer vectors,  can be viewed  as a rational perturbation of the original series  ${}_pF_{q}(\a;\b;x)$. Indeed, replacing a numerator parameter $a$ by $a+1$ amounts to multiplication of each term by the linear function $k\to (a+k)/a$ and this extends to larger shifts in an obvious way. Going from $a$ to $a-1$ amounts to division by a linear function. 

 Probably the first important example of a polynomially perturbed summation formula is the Karlsson-Minton summation theorem \cite{Karlsson,Minton}, which we write in the form
$$
F\!\left(\begin{matrix}a,b\\b+1\end{matrix}\:\bigg\vert\, P\,\bigg\vert\,1\right)=\frac{\Gamma(b+1)\Gamma(1-a)}
{\Gamma(b+1-a)}P(-b).
$$
This formula can be viewed as a polynomial perturbation of a particular case of the Gauss summation theorem to which it reduces when $P\equiv1$. The polynomial $P$ is usually written in terms of its negated zeros as follows:
$$
P(t)=\frac{(f_1+t)_{m_1}\cdots(f_r+t)_{m_r}}{(f_1)_{m_1}\cdots(f_r)_{m_r}},
$$
so that its degree is $m_1+m_2+\cdots+m_r$ and it is normalized by $P(0)=1$.  However, this particular way of writing $P$ is not important for the validity of the Karlsson-Minton summation theorem. A salient property of this theorem is that the same polynomial $P$ appears on both the left and the right hand sides. A characteristic feature of many hypergeometric identities discovered over last two decades is that a polynomial perturbation of the left hand side leads to a polynomial perturbation of the right hand side but with \emph{a different polynomial}. 
See our paper \cite{KPITSF2018} for extensions of the Karlsson-Minton theorem. 

A next significant development emerged through a series of contributions by various authors in the first decade of the 21st century culminating in the seminal paper by Miller and Paris \cite{MP2013}.  This work provides explicit formulas for the polynomial perturbations of the first and second Euler's (or Euler-Pfaff's) transformations for the Gauss hypergeometric function ${}_2F_1$ \cite[Theorem~2.2.5]{AAR}. As consequences the authors derive polynomial perturbations of Kummer's transformation for the confluent hypergeometric function ${}_1F_1$ and of some quadratic transformations for the Gauss function. Degenerate cases of Miller-Paris transformations have been considered by us in \cite{KPResults2019,KPChapter2020}. Further results for polynomial perturbations of the quadratic transformations were established by Xiaoxia Wang and Arjun Rathie  in \cite{WangRathie2013} and by Robert Maier in \cite{Maier2019}.  Polynomial perturbations have also been found for the classical transformations and summation formulas at fixed argument (typically $1$).  For example, Saalsch\"{u}tz's theorem was extended by Kim, Rathie and Paris in \cite{KRP2013}, while Whipple's and Dougall's identities have been extended by Srivastava, Vyas and Fatawat in \cite{SVF2019} and  Mishev in \cite{Mishev2022} (see also section~3 of this paper).  Numerous further examples of polynomial perturbations of summation and transformation identities at fixed argument can be found in our paper \cite{KPSymmetry2022}.   While  \cite{SVF2019} and \cite{Mishev2022} only involve linear and quadratic perturbations an important work \cite{VF2022} by Yashoverdhan Vyas and Kaplana Fatawat presents perturbations of the classical transformations and summation theorems by arbitrary polynomials - the results to play an important role in this paper.    

The main goal of this paper is to derive a polynomial perturbation of the celebrated 1828 Clausen's product identity, given by  \cite[p.116, 180]{AAR}: 
\begin{equation}\label{eq:Clausen}
\Bigg[{}_{2}F_{1}\!\left(\begin{matrix}c,d\\c+d+1/2\end{matrix}\,\bigg\vert\,x\right)\Bigg]^2={}_{3}F_{2}\!\left(\begin{matrix}2c,2d,c+d\\c+d+1/2,2c+2d\end{matrix}\,\bigg\vert\,x\right).   
\end{equation}
This identity played a crucial role in Ramanujan's derivation of some of his series for $\pi$.  A nice overview of its proof and relation to Ramanujan series can be found in Richard Askey's paper \cite{Askey1989}.  In fact, Askey makes one further step and finds an identity which can be viewed as quadratic perturbation of \eqref{eq:Clausen}, see \cite[(3.7)]{Askey1989} . More precisely, the right hand side of his formula is a quadratic perturbation of the right hand side of \eqref{eq:Clausen}, while the left hand side is contiguous to the left hand side of \eqref{eq:Clausen}.  Analogues of Clausen's identity of importance in number theory were discovered by Heng Huat Chan, Yoshio Tanigawa, Yifan Yang and Wadim Zudilin in \cite{CTYZ2011} using the theory of modular forms. Further investigations along these lines can be found in \cite{ASZ2011}.  In the same year 2011 as the above two references, another type of generalization of Clausen's identity \eqref{eq:Clausen} was published by Raimundas Vidunas in \cite{Vidunas2011}.  His generalizations are in terms of two-variate hypergeometric series and the bottom parameter of the ${}_2F_1$ function on the left hand side is unrestricted. 

In order to reach our goal of finding a polynomial perturbation of \eqref{eq:Clausen} we first revisit the perturbations of Euler-Pfaff's transformations discovered by Allen Miller and Richard Paris. Namely, we will present new and straightforward proofs for the Miller-Paris transformations and relate them to Newton and Lagrange interpolating polynomials. This will yield several new ways of writing the characteristic polynomials driving these transformations. Note that in our work \cite{KPChapter2020} we already found an alternative way to obtain Miller-Paris transformations and new forms of the characteristic polynomials. Those, however, are entirely different from what follows in this paper. We introduce a concept of the first and second Miller-Paris operators which will play a crucial role in generalizing Clausen's formula. This is done in the succeeding Section~2.   In Section~3,  we re-derive the formulas due to Vyas and Fatawat from \cite{VF2022} to reveal explicitly the underlying characteristic polynomials somewhat hidden in their work.  These formulas are then used to extend Clausen's identity as stated in Theorem~\ref{th:ClausenPerturbed}.  In the two corollaries of this theorem we give explicit forms for the linear and quadratic perturbations.  In the final Section~4, we present remarks on open problems related to the new identity and outline potential directions for future work.
	
	\bigskip
	\bigskip
	
	\section{Polynomial perturbations of Euler's transformations}
	We will need a particular form of the Lagrange and Newton interpolating polynomials for the nodes of interpolation at $0,1,\ldots,m$. They are  given in the following lemma.
	\begin{lemma}\label{lm:interpolation}
		Given any real numbers $a_0,\ldots,a_m$ the polynomial $P_{m}$ of degree $m$ taking the values $P_{m}(j)=a_j$, $j=0,\ldots,m$, is given by \emph{(}Lagrange's form\emph{)}
	\begin{equation}\label{eq:LagrangeP}
			P_{m}(t)=\frac{(-1)^m}{m!}\sum\limits_{k=0}^{m}\binom{m}{k}(-t)_k(t-m)_{m-k}a_k=\frac{(-t)_{m+1}}{m!}\sum\limits_{k=0}^{m}\binom{m}{k}\frac{(-1)^ka_k}{k-t},
		\end{equation}
		or \emph{(}Newton's form\emph{)}
	\begin{equation}\label{eq:NewtonP}
			P_{m}(t)=\sum\limits_{k=0}^{m}\frac{(-t)_{k}}{k!}\sum\limits_{j=0}^{k}(-1)^j\binom{k}{j}a_j.
		\end{equation}
	\end{lemma}
	\begin{proof} Indeed, $(-j)_k=0$ for $j=0,1,\ldots,k-1$ and $(j-m)_{m-k}=0$ for $j=k+1,\ldots, m$. Hence,
	$$
	P_{m}(j)=\frac{(-1)^m}{m!}\binom{m}{j}(-j)_j(j-m)_{m-j}a_j=a_j.
	$$
    Formula \eqref{eq:NewtonP} is the classical Newton's form.
\end{proof}	
	
	\subsection{The first Miller-Paris transformation}\label{sbsc:MP1}
	
	For any $\alpha$ and any numerical sequence $a_k$ Leonard Euler found the following power series transformation understood either formally or, in case of convergence, as equality of functions:
	$$
\sum\limits_{k=0}^{\infty}\frac{(\alpha)_{k}}{k!}a_kx^k=(1-x)^{-\alpha}\sum\limits_{k=0}^{\infty}
	\frac{(\alpha)_{k}}{k!}\Delta^k{a_0}\:\Big(\frac{x}{1-x}\Big)^k,
$$
	where the forward difference is defined by  $\Delta{a_j}=a_{j+1}-a_{j}$, $\Delta^{k}{a_j}:=\Delta(\Delta^{k-1}{a_{j}})$. Taking 
$$
	a_k=\frac{(\a)_k}{(\b)_k}:=\frac{(a_1)_{k}\cdots(a_p)_{k}}{(b_1)_{k}\cdots(b_p)_{k}}
$$
it is easy to verify that 
$$
	\Delta^k{a_0}=(-1)^k{}_{p+1}F_{p}\!\left(\begin{matrix}-k,\a\\\b\end{matrix}\right):=(-1)^k{}_{p+1}F_{p}\!\left(\begin{matrix}-k,\a\\\b\end{matrix}\,\bigg\vert\,1\right)
$$
(the argument $1$ will be omitted from here on).  This leads to N{\o}rlund's formula \cite[(1.21)]{Norlund} 
\begin{equation}\label{eq:Norlund1.1.21}
		{}_{p+1}F_{p}\!\left(\left.\begin{matrix}\alpha,\a\\\b\end{matrix}\:\right\vert x\right)=(1-x)^{-\alpha}\sum\limits_{k=0}^{\infty}\frac{(\alpha)_{k}}{k!}{}_{p+1}F_{p}\!\left(\begin{matrix}-k,\a\\\b\end{matrix}\right)
		\Big(\frac{x}{x-1}\Big)^k.
	\end{equation}
	We will apply this relation  to a particular case of the generalized hypergeometric function given by
\begin{equation}\label{eq:rFrFm}
{}_{r+1}F_{r}\!\left(\begin{matrix}a,b,\f+\m\\c,\f\end{matrix}\:\Big\vert x\right)
=\sum\limits_{k=0}^{\infty}\frac{(a)_k(b)_k}{(c)_kk!}F_m(k)x^k=:F\!\left(\begin{matrix}a,b\\c\end{matrix}\:\Big\vert\, F_m\,\Big\vert\,x\right),
\end{equation}
	where $\f=(f_1,\ldots,f_r)\in(\C\setminus{-\N_0})^{r}$, $\m=(m_1,\ldots,m_r)\in\N^r$, and $F_{m}(k)$ is a polynomial of degree $m=m_1+m_2+\cdots+m_r$ defined by 
	\begin{equation}\label{eq:Fm-defined}
		F_m(k)=\frac{(\f+\m)_{k}}{(\f)_{k}}=\frac{(\f+k)_{\m}}{(\f)_{\m}}=\frac{1}{(\f)_{\m}}\sum_{j=0}^{m}\sigma_j(\f,\m)k^{j},
	\end{equation}
    where from here onward we will use the shorthand notation $(\f)_{\m}=(f_1)_{m_1}\cdots(f_r)_{m_r}$ and $(\f+k)=(f_1+k,f_2+k,\ldots,f_r+k)$ and similarly for other vectors.
	The ultimate equality in \eqref{eq:Fm-defined} is the definition of the numbers $\sigma_j(\f,\m)$. Assuming that $(c-b-m)_{m}\ne0$ and $b=f_{j}$ for all $j$ and applying  \eqref{eq:Norlund1.1.21} to the function  \eqref{eq:rFrFm} we obtain
\begin{align}\label{eq:MP1general}
F\!\left(\begin{matrix}a,b\\c\end{matrix}\:\bigg\vert\,F_m\,\bigg\vert\, x\right)&\!=\frac{1}{(1-x)^{a}}\sum\limits_{n=0}^{\infty}\frac{(a)_{n}(c-b-m)_n}{(c)_nn!}A_n
=\frac{1}{(1-x)^{a}}F\!\left(\begin{matrix}a,c-b-m\\c\end{matrix}\:\bigg\vert\,Q_m\,\bigg\vert\,\frac{x}{x-1}\right),
		\\\nonumber
		~\text{where}~A_n&=\frac{(c)_n}{(c-b-m)_n}{}_{r+2}F_{r+1}\!\left(\begin{matrix}-n,b,\f+\m\\c,\f\end{matrix}\right)=Q_{m}(n),~n=0,1,\ldots,
\end{align}
	with $Q_m$ defined by 
	$$
	Q_m(t)=\frac{\Gamma(c+t)\Gamma(c-b-m)}{\Gamma(c)\Gamma(c-b-m+t)}{}_{r+2}F_{r+1}\!\left(\begin{matrix}-t,b,\f+\m\\c,\f\end{matrix}\right).
	$$
We will show next that $Q_m(t)$ is, in fact, a polynomial in $t$ of degree $m$. This can be easily seen once the first Miller-Paris transformation is already known, see \cite[Theorem~4]{MP2012}.  It was also proved in \cite[Theorem~3.2]{KPITSF2018}, but we prefer to give an independent and more direct proof here.  Using \eqref{eq:Fm-defined} we get
$$
{}_{r+2}F_{r+1}\!\left(\begin{matrix}t,b,\f+\m\\c,\f\end{matrix}\right)=\frac{1}{(\f)_{\m}}\sum_{j=0}^{m}\sigma_j(\f,\m)\sum\limits_{k=0}^{\infty}\frac{(t)_k(b)_k}{(c)_kk!}k^{j}.
$$
Denote $D=\partial_{x}$. We will need the following easily verifiable expansion of the differential operator $(xD)^n$, 
	$$
(xD)^n=\sum\limits_{i=1}^{n}S(n,i)x^{i}D^{i}
	$$
in terms of Stirling's numbers of the second kind $S(n,i)$ generated by $x^n=\sum_{i=1}^{n}S(n,i)[x]_{i}$, where $[x]_i=x(x-1)\cdots(x-i+1)$ is the falling factorial. Then 
	\begin{multline*}
		(xD)^{j}{}_{2}F_{1}\!\left(\left.\begin{matrix}t,b\\c\end{matrix}\:\right|\:x\right)=\sum\limits_{k=0}^{\infty}\frac{(t)_k(b)_k}{(c)_kk!}k^{j}x^k=\sum\limits_{i=1}^{j}S(j,i)x^{i}D^{i}{}_{2}F_{1}\!\left(\left.\begin{matrix}t,b\\c\end{matrix}\:\right|\:x\right)
		\\
		=
		\sum\limits_{i=1}^{j}S(j,i)x^{i}\frac{(t)_i(b)_i}{(c)_i}{}_{2}F_{1}\!\left(\left.\begin{matrix}t+i,b+i\\c+i\end{matrix}\:\right|\:x\right) .   
	\end{multline*}
	Setting $x=1$ and using the Gauss summation theorem \cite[Theorem~2.2.2]{AAR} we have
$$
\sum\limits_{k=0}^{\infty}\frac{(t)_k(b)_k}{(c)_kk!}k^{j}=\sum\limits_{i=1}^{j}S(j,i)\frac{(t)_i(b)_i}{(c)_i}{}_{2}F_{1}\!\left(\left.\begin{matrix}t+i,b+i\\c+i\end{matrix}\:\right|\:1\right)= 
	\sum\limits_{i=1}^{j}S(j,i)\frac{(t)_i(b)_i}{(c)_i}\frac{\Gamma(c+i)\Gamma(c-t-b-i)}{\Gamma(c-t)\Gamma(c-b)}.
	$$
Substituting this into the definition of $Q_m$ we obtain: 
	\begin{multline*}
		Q_m(-t)=\frac{\Gamma(c-t)\Gamma(c-b-m)}{\Gamma(c)\Gamma(c-b-m-t)}{}_{r+2}F_{r+1}\!\left(\begin{matrix}t,b,\f+\m\\c,\f\end{matrix}\right)
		\\
		=\frac{\Gamma(c-t)\Gamma(c-b-m)}{\Gamma(c)\Gamma(c-b-m-t)(\f)_{\m}}\sum_{j=0}^{m}\sigma_j(\f,\m)\sum\limits_{k=0}^{\infty}\frac{(t)_k(b)_k}{(c)_kk!}k^{j}
		\\
		=\frac{\Gamma(c-t)\Gamma(c-b-m)}{\Gamma(c)\Gamma(c-b-m-t)(\f)_{\m}}\sum_{j=0}^{m}\sigma_j(\f,\m)\sum\limits_{i=1}^{j}S(j,i)\frac{(t)_i(b)_i}{(c)_i}\frac{\Gamma(c+i)\Gamma(c-t-b-i)}{\Gamma(c-t)\Gamma(c-b)}
		\\
		=\frac{(-1)^m}{(1+b-c)_{m}(\f)_{\m}}\sum_{j=0}^{m}\sigma_j(\f,\m)\sum\limits_{i=1}^{j}S(j,i)(t)_i(b)_i(c-t-b-m)_{m-i},
	\end{multline*}
	which is manifestly a polynomial of degree $m$ in $t$.  Hence, $Q_m(t)$ coincides with  interpolating polynomial taking the values $A_n$ at $t=n$ for $n=0,\ldots,m$. We can use an explicit construction from Lemma~\ref{lm:interpolation}, for example:
	\begin{equation}\label{eq:LagrangeQm}
		Q_m(t)=\frac{(-1)^m}{m!}\sum\limits_{k=0}^{m}\binom{m}{k}\frac{(c)_k(-t)_k(t-m)_{m-k}}{(c-b-m)_k}{}_{r+2}F_{r+1}\!\left(\begin{matrix}-k,b,\f+\m\\c,\f\end{matrix}\right).
	\end{equation}
Note that all forms for $Q_m$ provided by Lemma~\ref{lm:interpolation} differ from that of Miller and Paris, which is
\begin{equation}\label{eq:MPQm}
		Q_m(t)=\frac{1}{(c-b-m)_{m}}\sum\limits_{k=0}^{m}\frac{(b)_k(-t)_{k}(c-b-m+t)_{m-k}}{(-1)^kk!}{}_{r+1}F_{r}\!\left(\begin{matrix}-k,\f+\m\\\f\end{matrix}\right),
\end{equation}
as well as from the alternative form given by the author and E.G.\:Prilepkina in \cite[Lemma~2]{KPChapter2020}, namely,
	\begin{equation}\label{eq:KPQm}
		Q_m(t)=\frac{(\f-b)_{\m}}{(\f)_{\m}(c-b-m)_{m}}\sum\limits_{k=0}^{m}\frac{(b)_k(1-c+b)_{k}(c-b-m+t)_{m-k}}{(-1)^kk!}{}_{r+1}F_{r}\!\left(\begin{matrix}-k,1-\f+b\\1-\f+b-\m\end{matrix}\right).
	\end{equation}
Note that both $F_m$ and $Q_m$ are normalized by $Q_m(0)=F_m(0)=1$.
	
	Further, by changing the order of summation in \eqref{eq:MPQm} and application of \cite[3.2(7)]{LukeBookVol1}, it is not hard to express $Q_m$ in terms of values of the polynomial $F_m$ at nonnegative integers as follows 
	\begin{align}\label{eq:MPQmNew}
		Q_m(t)=&\frac{(-1)^m(-t)_{m}(b)_m}{(c-b-m)_{m}}\sum\limits_{k=0}^{m}\frac{(-1)^k}{k!(m-k)!}F_m(k){}_{3}F_{2}\!\left(\begin{matrix}-m+k,1,c-b-m+t\\1-m+t,1-m-b\end{matrix}\right)
		\\
		=&\frac{(-t)_{m}(b)_m}{(c-b-m)_{m}m!}\Delta^m\bigg[F_m(k){}_{3}F_{2}\!\left(\begin{matrix}-m+k,1,c-b-m+t\\1-m+t,1-m-b\end{matrix}\right)\bigg]_{k=0},
	\end{align}
where the action of the operator  $\Delta^m$ is in the variable $k$ and one should set $k=0$ after applying  $\Delta^m$.  Further, by an application of \cite[Appendix (II)]{RJRJR1992} to ${}_3F_2$ on the right hand side we also get 
	\begin{equation}\label{eq:MPQmNewNew}
		Q_m(t)=\frac{(b)_{m}(c)_m}{(1-c+b)_{m}}\sum\limits_{k=0}^{m}\frac{(-1)^k(-t)_k}{k!(m-k)!(c)_k}F_m(k){}_{3}F_{2}\!\left(\begin{matrix}-m+k,1-c-t,-b-m\\1-c-m,1-b-m\end{matrix}\right).
	\end{equation}
	
	The first Miller-Paris transformation \eqref{eq:MP1general}  can be viewed as a mapping from the polynomial $F_m$ to the polynomial $Q_m$.  This way of looking at Miller-Paris transformations clarifies many further considerations, so  we will frequently write 
	\begin{equation}\label{eq:Tm-defined}
		Q_m(t)=	Q_m\!\left(\begin{matrix}b\\c\end{matrix}\:\bigg\vert\,F_m\,\bigg\vert\, t\right)=[T_m(b;c)F_{m}(\cdot)](t)=T_m(b;c)F_{m}(\cdot),
	\end{equation}
where $T_m(b;c)$ will be designated \textit{the first Miller-Paris operator}. The name of the argument of the resulting polynomial $Q_m$ will sometimes be omitted on the right hand side as in the rightmost expression above.  The operator $T_m(b;c)$ depends on the degree $m$ and the two parameters $b$ and $c$ and acts on polynomials according to any of the equivalent formulas \eqref{eq:LagrangeQm}-\eqref{eq:MPQmNewNew}.  Moreover, formulas \eqref{eq:MPQmNew} and \eqref{eq:MPQmNewNew} show that its action is well defined on any finite sequence indexed by $k=0,\ldots,m$, and it is linear for each fixed $m$.  Furthermore,  
    the second application of the transformation \eqref{eq:MP1general} to the right hand side of the same formula with $a$ fixed and $c-b-m$ playing the role of $b$ clearly recovers the left hand side. This implies that the inverse operator to  $T_m(b;c)$ is given by
	\begin{equation}\label{eq:Tm-inverse}
		[T_m(b;c)]^{-1}=T_m(c-b-m;c).
	\end{equation}
This identity written explicitly using \eqref{eq:MPQmNew} or \eqref{eq:MPQmNewNew} represents a combinatorial type inversion formula, which might be new.

\subsection{The second Miller-Paris transformation}\label{sbsc:MP2}

Assume that $(c-b-m)_m\ne0$, $(c-a-m)_m\ne0$ and $(1+a+b-c)_m\ne0$. Second application of the transformation \eqref{eq:MP1general} to its right hand side, this time leaving the parameter $c-b-m$ intact, leads immediately to the following identity  
	\begin{equation}\label{eq:MP2general}
			F\!\left(\begin{matrix}a,b\\c\end{matrix}\,\bigg\vert\,F_m\,\bigg\vert\,x\right)\!=\!{}_{r+2}F_{r+1}\left.\left(\begin{matrix}a, b,\f+\m\\c,\f\end{matrix}\,\right\vert x\right)
		\!=\!(1-x)^{c-a-b-m}F\left(\!\begin{matrix}c-a-m, c-b-m\\c\end{matrix}\,\bigg\vert\,\hat{Q}_m\,\bigg\vert\,x\right),
	\end{equation}
	where the polynomial $\hat{Q}_m$ results from the application of the operator $T_m(a;c)$ defined in \eqref{eq:Tm-defined} to the polynomial $Q_m$, i.e.  $\hat{Q}_m=T_m(a;c)Q_m=T_m(a;c)T_m(b;c)F_m$.  Various explicit forms for $\hat{Q}_m$  can be obtained by combining formulas \eqref{eq:LagrangeQm}-\eqref{eq:MPQmNewNew} when computing the action of the operators $T_m(b;c)$, $T_m(a;c)$.  A simpler expression can be found by using interpolation.  Indeed, multiplying the ${}_{r+2}F_{r+1}$ on the left hand side by $(1-x)^{a+b+m-c}$ expanded in binomial series, by Cauchy product we obtain
\begin{multline}\label{eq:MP2CauchyProduct}
		(1-x)^{a+b+m-c}{}_{r+1}F_{r}\left.\!\!\left(\!\begin{matrix}a, b,\f+\m\\c,\f\end{matrix}\right\vert x\right)
		=\sum\limits_{n=0}^{\infty}x^n\sum\limits_{k=0}^{n}
		\frac{(a)_k(b)_k(\f+\m)_{k}(c-a-b-m)_{n-k}}{(c)_{k}(\f)_kk!(n-k)!}
		\\
		=\sum\limits_{n=0}^{\infty}x^n
		\frac{(c-a-b-m)_{n}}{n!}
		{}_{r+3}F_{r+2}\!\left(\begin{matrix}-n,a,b,\f+\m\\c,1-c+a+b+m-n,\f\end{matrix}\right)
        \\
		=\sum\limits_{n=0}^{\infty}x^n
		\frac{(c-a-m)_{n}(c-b-m)_{n}}{(c)_{n}n!}B_n,
	\end{multline}
	where $(\alpha)_{n-k}=(-1)^k(\alpha)_{n}/(1-\alpha-n)_{k}$ has been used, and we defined 
	\begin{equation}\label{eq:Bn}
		B_n:=\frac{(c-a-b-m)_{n}(c)_{n}}{(c-a-m)_{n}(c-b-m)_{n}}
		{}_{r+3}F_{r+2}\!\left(\begin{matrix}-n,a,b,\f+\m\\c,1-c+a+b+m-n,\f\end{matrix}\right).    
	\end{equation}
	Now, we can apply Lemma~\ref{lm:interpolation} to build a polynomial $P_{m}(t)$ having the values $P_{m}(n)=B_n$ for $n=0,1,\ldots,m$.  However, as both $P_m$ and $\hat{Q}_m$ have degree $m$ and agree at $m+1$ points, they must coincide identically.  For instance, formula \eqref{eq:LagrangeP} yields:
\begin{equation}\label{eq:hatQLagrange}
		\hat{Q}_m(t)=\frac{(-1)^m}{m!}\sum\limits_{k=0}^{m}\binom{m}{k}\frac{(-t)_k(t\!-\!m)_{m-k}(c\!-\!a\!-\!b\!-\!m)_{k}(c)_{k}}{(c-a-m)_{k}(c-b-m)_{k}}
		{}_{r+3}F_{r+2}\!\left(\begin{matrix}-k,a,b,\f+\m\\c,1\!-\!c\!+\!a\!+\!b\!+\!m\!-\!k,\f\end{matrix}\right). 
	\end{equation}
	The form of the polynomial $\hat{Q}_m$ found by Miller and Paris \cite[Theorem~4]{MP2013}  is different, namely,
	\begin{equation}\label{eq:hatQMP}
		\hat{Q}_m(t)=\sum\limits_{k=0}^{m}\frac{(a)_k(b)_k(-t)_k}{k!(c-a-m)_k(c-b-m)_k} {}_{r+1}F_{r}\!\left(\begin{matrix}-k,\f+\m\\\f\end{matrix}\right)
		{}_{3}F_{2}\!\left(\begin{matrix}-m+k,k-t,c-a-b-m\\c-a-m+k,c-b-m+k\end{matrix}\right).
	\end{equation}
	Below will show that this form is a guise of Newton's interpolation formula  \eqref{eq:NewtonP}. 
	
	Now we will give an alternative direct proof of the transformation formula \eqref{eq:MP2general} without resorting to the first Miller-Paris transformation \eqref{eq:MP1general}. To this end we will show that the sequence $B_n$ defined in \eqref{eq:Bn} has vanishing forward differences of all orders exceeding $m$, which implies that it can be interpolated by a polynomial of degree $m$ for all integer $n\ge0$. This is done in the following lemma. 
	\begin{lemma}\label{lm:Bndiff}
		Suppose the sequence $B_n$ is defined by \eqref{eq:Bn}. Then, for all integer $n>m$ we have 
		\begin{equation}\label{eq:nthDiffB_n}
			\Delta^{n}B_0=(-1)^n\sum\limits_{k=0}^{n}(-1)^k\binom{n}{k}B_k=0.
		\end{equation}
	\end{lemma}
\begin{proof} Substituting the definition of $B_k$ from \eqref{eq:Bn}, writing the hypergeometric function as a sum and exchanging the  order of summations after some calculations we will have 
	$$
	\sum\limits_{k=0}^{n}(-1)^k\binom{n}{k}B_k=\sum\limits_{j=0}^{n}\binom{n}{j}\frac{(-1)^j(a)_j(b)_j(\f+\m)_{j}}{(c-a-m)_{j}(c-b-m)_{j}(\f)_{j}}
	{}_{3}F_{2}\!\left(\begin{matrix}-n+j,c+j,c-a-b-m\\c-a-m+j,c-b-m+j\end{matrix}\right),
	$$
	where we applied standard transformations of Pochhammer's symbols $(\alpha)_{k+j}=(\alpha)_{j}(\alpha+j)_{k}$, $(\alpha-j)_{j}=(-1)^{j}(1-\alpha)_{j}$ and $(n-k-j)!=(-1)^k(n-j)!/(-n+j)_{k}$.  Next,  apply one of the Sheppard's transformations as given in \cite[Appendix, formula (III)]{RJRJR1992} to the ${}_3F_2$ on the right hand side. With appropriate change of notation this transfromation is given by
	\begin{align*}
		&{}_{3}F_{2}\!\left(\begin{matrix}-n+j,c+j,c-a-b-m\\c-a-m+j,c-b-m+j\end{matrix}\right)
		\\
		&=
		\frac{(a)_n(b)_n(c-a-m)_{j}(c-b-m)_{j}}{(c-a-m)_{n})(c-b-m)_{n}(a)_{j}(b)_{j}}{}_{3}F_{2}\!\left(\begin{matrix}1+m-n,-n+j,c-a-b-m\\1-a-n,1-b-n\end{matrix}\right).    
	\end{align*}
	Substituting this back into the above formula and using $(\f+\m)_{j}/(\f)_{j}=(\f+j)_{\m}/(\f)_{\m}$ yields
	$$
	\sum\limits_{k=0}^{n}(-1)^k\binom{n}{k}B_k=\frac{(a)_n(b)_n}{(c-a-m)_{n})(c-b-m)_{n}(\f)_{\m}}\sum\limits_{j=0}^{n}(-1)^j\binom{n}{j}G(j),
	$$
	where 
	$$
	G(t)=(\f+t)_{\m}
	\cdot{}_{3}F_{2}\!\left(\begin{matrix}1+m-n,-n+t,c-a-b-m\\1-a-n,1-b-n\end{matrix}\right)
	$$
	is manifestly  a polynomial in $t$ of degree $n-1$.  Hence,
$\Delta^{n}G(0)=0$.
\end{proof}

	This lemma implies that any polynomial that interpolates $B_0, B_1,\ldots,B_m$ (i.e. $P(k)=B_k$ for $k=0,\ldots,m$) actually interpolates $B_k$ for all integers $k\ge0$.  Indeed, suppose we build Newton's interpolating polynomial of degree $N>m$ taking the values $P(k)=B_k$ for $k=0,\ldots,N$ as prescribed by  \eqref{eq:NewtonP}.  Then the inner sum in \eqref{eq:NewtonP} vanishes for $k>m$ by the above lemma and the polynomial has degree $m$.  As this is true for any $N>m$, the conclusion follows. 
	
Hence, we established \eqref{eq:MP2general} with characteristic polynomial $\hat{Q}_m(t)$ given in \eqref{eq:hatQLagrange}.  This does not explain, however, where Miller-Paris' form of  $\hat{Q}_m(t)$ in \eqref{eq:hatQMP} comes from. We will show that, in fact, \eqref{eq:hatQMP} is a guise of Newton's interpolating polynomial \eqref{eq:NewtonP}.
	First, we express $\hat{Q}_m(t)$ in terms of the rising factorial basis $(-t)_n$.  To this end expand ${}_3F_2$ in \eqref{eq:hatQMP} and exchange the order of summations using $(\alpha)_{k}(\alpha+k)_{j}=(\alpha)_{k+j}$:
\begin{multline}\label{eq:hatQMPrearranged}
		\hat{Q}_m(t)=\sum\limits_{k=0}^{m}\frac{(-1)^kC_{k,r}(a)_k(b)_k}{(c-a-m)_k(c-b-m)_k}
		\sum\limits_{j=0}^{m-k}\frac{(-m+k)_{j}(c-a-b-m)_{j}(-t)_k(k-t)_{j}}{(c-a-b-m+k)_{j}(c-b-m+k)_{j}j!}
		\\
		=\sum\limits_{n=0}^{m}\frac{(-m)_n(-t)_n}{(c-a-m)_{n}(c-b-m)_{n}n!}
		\sum\limits_{k+j=n}\binom{n}{k}\frac{(a)_k(b)_k(c-a-b-m)_{j}}{(-m)_{k}}{}_{r+1}F_{r}\!\left(\begin{matrix}-k,\f+\m\\\f\end{matrix}\right).
\end{multline}
This is an expansion of $\hat{Q}_m(t)$ in the rising factorial basis $(-t)_{n}$.
On the other hand, substituting the numbers $a_k=B_k$ from \eqref{eq:Bn} into \eqref{eq:NewtonP} we arrive at the following formula for the coefficient at $(-t)_n$:
\begin{multline}\label{eq:Newton-MP}
	\sum\limits_{k=0}^{n}(-1)^k\binom{n}{k}\frac{(c-a-b-m)_{k}(c)_{k}}{(c-a-m)_{k}(c-b-m)_{k}}
		{}_{r+3}F_{r+2}\!\left(\begin{matrix}-k,a,b,\f+\m\\c,1-c+a+b+m-k,\f\end{matrix}\right).
\end{multline}
Below we will give a direct hypergeometric proof that the coefficient at  $(-t)_n$ in \eqref{eq:hatQMPrearranged}  coincides with  \eqref{eq:Newton-MP} .   Note also that \cite[Corollary~3]{MP2012} can be recovered from this equality by setting $c=b+1$ and performing some simplifications. 
\begin{lemma}\label{lm:CharacterPol}
The coefficient at  $(-t)_n$ in \eqref{eq:hatQMPrearranged} is equal to \eqref{eq:Newton-MP} for all $n=0,1,\ldots$, whereas for $n>m$ both of them vanish. 
\end{lemma}
	\begin{proof}  We will use the obvious relation 
	$(\alpha)_{n-k}=(-1)^{k}(\alpha)_{n}/(1-\alpha-n)_{k}$ repeatedly without further mentioning.  Start with the inner sum in \eqref{eq:hatQMPrearranged}. Expanding the hypergeometric function and exchanging the order of summations we obtain:
	\begin{multline*}
		\sum\limits_{k=0}^{n}\binom{n}{k}\frac{(a)_k(b)_k(c-a-b-m)_{n-k}}{(-m)_{k}}\sum\limits_{j=0}^{k}\frac{(-k)_{j}(\f+\m)_{j}}{(\f)_{j}j!}
		\\
		=\sum\limits_{j=0}^{n}\sum\limits_{k=j}^{n}
		\frac{(-1)^k(-k)_{j}(\f+\m)_{j}(a)_k(b)_k(-n)_{k}(c-a-b-m)_{n-k}}{(\f)_{j}j!k!(-m)_{k}}
		\\
		=[l:=k-j]=\sum\limits_{j=0}^{n}\sum\limits_{l=0}^{n-j}
		\frac{(-1)^{l+j}(-l-j)_{j}(\f+\m)_{j}(a)_{l+j}(b)_{l+j}(-n)_{l+j}(c-a-b-m)_{n-l-j}}{(\f)_{j}j!(1)_{l+j}(-m)_{l+j}}
		\\
		=\sum\limits_{j=0}^{n}\sum\limits_{l=0}^{n-j}
		\frac{(\f+\m)_{j}(a)_{j}(a+j)_{l}(b)_{j}(b+j)_{l}(-n)_{j}(-n+j)_{l}(c-a-b-m)_{n-j}}{(1-c+a+b+m-n+j)_{l}(\f)_{j}j!(1)_{l}(-m)_{j}(-m+j)_{l}}
		\\
		=\sum\limits_{j=0}^{n}
		\frac{(-n)_{j}(a)_{j}(b)_{j}(c-a-b-m)_{n-j}(\f+\m)_{j}}{(-m)_{j}(\f)_{j}j!}
		{}_{3}F_{2}\!\left(\begin{matrix}-n+j,a+j,b+j\\-m+j,1-c+a+b+m-n+j\end{matrix}\right).
	\end{multline*}
Now consider  \eqref{eq:Newton-MP}. Similar rearrangement yields:
	\begin{multline*}
		\sum\limits_{k=0}^{n}\frac{(-n)_{k}(c-a-b-m)_{k}(c)_{k}}{(c-a-m)_{k}(c-b-m)_{k}k!}
		\sum\limits_{j=0}^{k}\frac{(-k)_j(a)_{j}(b)_{j}(\f+\m)_{j}(-1)^{j}(c-a-b-m)_{k-j}}{(c)_{j}(\f)_{j}(c-a-b-m)_{k}j!}
		\\
		=\sum\limits_{j=0}^{n}\sum\limits_{k=j}^{n}\frac{(a)_{j}(b)_{j}(\f+\m)_{j}(-1)^{j}(-n)_{k}(c)_{k}(-k)_j(c-a-b-m)_{k-j}}{(c)_{j}(\f)_{j}j!(c-a-m)_{k}(c-b-m)_{k}k!}
		\\
		=\sum\limits_{j=0}^{n}\sum\limits_{l=0}^{n-j}\frac{(a)_{j}(b)_{j}(\f+\m)_{j}(-n)_{j+l}(c)_{j+l}(c-a-b-m)_{l}}{(c)_{j}(\f)_{j}j!(c-a-m)_{j+l}(c-b-m)_{j+l}(1)_{l}}
		\\
		=\sum\limits_{j=0}^{n}\frac{(-n)_{j}(a)_{j}(b)_{j}(\f+\m)_{j}}{(c-a-m)_{j}(c-b-m)_{j}(\f)_{j}j!}{}_{3}F_{2}\!\left(\begin{matrix}-n+j,c+j,c-a-b-m\\c-a-m+j,c-b-m+j\end{matrix}\right).
	\end{multline*}
	To establish the equality it remains to note that 
	\begin{multline*}
		\frac{(c-a-m+j)_{n-j}(c-b-m+j)_{n-j}}{(-m+j)_{n-j}(c-a-b-m)_{n-j}}{}_{3}F_{2}\!\left(\begin{matrix}-n+j,c+j,c-a-b-m\\c-a-m+j,c-b-m+j\end{matrix}\right)
		\\
		={}_{3}F_{2}\!\left(\begin{matrix}-n+j,a+j,b+j\\-m+j,1-c+a+b+m-n+j\end{matrix}\right)    
	\end{multline*}
	according to \cite[Appendix, formula (VI)]{RJRJR1992}.
	\end{proof}
Yet another, somewhat simpler form of the polynomial $\hat{Q}_m$ was obtained in \cite[Theorem~2]{KPChapter2020} as follows:
\begin{equation}\label{eq:hatQKP}
		\hat{Q}_m(t)=\sum\limits_{k=0}^m\frac{(-1)^k(a)_k(-b-m)_k(-t)_k(c-a-m+t)_{m-k}}{(c-a-m)_m(c-b-m)_kk!}
		{}_{r+2}F_{r+1}\!\left(\begin{matrix}-k,b,\f+\m\\b+m-k+1,\f\end{matrix}\right).
\end{equation}
Changing the order of summation in this formula, rearranging Pochhammer's symbols and applying \cite[3.2(7)]{LukeBookVol1} for the  truncated hypergeometric series, we get
\begin{equation}\label{eq:hatQKPNew}
\hat{Q}_m(t)\!=\!\frac{(a)_m(b)_{m+1}(-t)_m}{(1-c+a)_{m}(1-c+b)_{m}m!}\sum\limits_{k=0}^m(-1)^k\binom{m}{k}\frac{F_m(k)}{b+k}
{}_{4}F_{3}\!\left(\begin{matrix}-m+k,1,1\!-\!c\!+\!b,c\!-\!a\!+\!t\!-\!m\\1-a-m,1+t-m,b+1+k\end{matrix}\right).
\end{equation}
Note that the ${}_4F_3$ function on the right hand side is Saalsch\"{u}tzian (i.e. the sum of the top parameters is one less than the sum of the bottom parameters), so that Bailey's transformations as summarized in \cite[Appendix~2]{RDN2002} are applicable.  For instance,  application of \cite[Appendix~2(I)]{RDN2002} yields the  form symmetric in $a$, $b$:
	\begin{multline}\label{eq:hatQKPNewNew}
		\hat{Q}_m(t)=\frac{(a)_{m+1}(b)_{m+1}(-t)_{m+1}}{(1-c+a)_{m}(1-c+b)_{m}m!}\sum\limits_{k=0}^m\binom{m}{k}\frac{(-1)^kF_m(k)}{(a+k)(b+k)(-t+k)}
		\\
		\times{}_{4}F_{3}\!\left(\begin{matrix}-m+k,1,c+k,a+b-c-t+m+k+1\\a+k+1,b+k+1,-t+k+1\end{matrix}\right).
	\end{multline}

The operator mapping the polynomial $F_m$ to the polynomial $\hat{Q}_m$ will be denoted by  $\hat{T}_m(a,b;c)$ and  designated \textit{the second Miller-Paris operator}. It depends on the degree $m$,  symmetric with respect to permutation of $a$ and $b$, and acts on polynomials  (or any sequences of length $m+1$) according to any of the equivalent formulas  \eqref{eq:hatQLagrange}, \eqref{eq:hatQMP}, \eqref{eq:hatQKP}, \eqref{eq:hatQKPNew}. In terms of the first Miller-Paris operator defined in \eqref{eq:Tm-defined} we have
	\begin{equation}\label{eq:B-decomposed}
		\hat{T}_m(a,b;c)=T_{m}(a;c)T_{m}(b;c)=T_{m}(b;c)T_{m}(a;c).
	\end{equation}
Hence, we can write 
	\begin{equation}\label{eq:Bm-defined}	\hat{Q}_m(t)=\hat{Q}_m\!\left(\begin{matrix}a,b\\c\end{matrix}\,\bigg\vert\,F_m\,\bigg\vert\,t\right)=[\hat{T}_m(a,b;c)F_{m}](t)=T_{m}(a;c)Q_m\!\left(\begin{matrix}b\\c\end{matrix}\,\bigg\vert\,F_m\,\bigg\vert\, \cdot\right),
	\end{equation}
	  From \eqref{eq:Tm-inverse}, we conclude that 
	$$
	[\hat{T}_m(a,b;c)]^{-1}=\hat{T}_m(c-a-m,c-b-m;c).
	$$
Similarly to \eqref{eq:Tm-inverse}, this identity written explicitly using \eqref{eq:hatQKPNew} or  \eqref{eq:hatQKPNewNew} represents a combinatorial type inversion formula, which might be new.

\section{Polynomial perturbation of Clausen's product identity}

By multiplying the power series coefficients of each of the ${}_2F_1$ functions on the left hand side of Clausen's identity \eqref{eq:Clausen} by the same polynomial $k\to(\f+k)_{\m}/(\f)_{\m}$ we get a polynomial perturbation of the left hand side.  The main result of this section asserts that the function we obtain on the right hand side remains hypergeometric and represents a rational perturbation of the ${}_3F_2$ function on the right hand side of \eqref{eq:Clausen}. We furthermore give an explicit form of this function in the following main theorem.

\begin{theorem}\label{th:ClausenPerturbed}
For the values of parameters such that both sides make sense, the following product formula holds
\begin{equation}\label{eq:ClausenPerturbed}
\Bigg[{}_{r+2}F_{r+1}\left(\begin{matrix}
b, c, \f+\m  \\
1/2+c+b+m, \f
\end{matrix}\bigg\vert\:x\right)\bigg]^2
\!\!=\!\frac{(1/2)_{m}(c+b+1/2)_{m}}{(b+1/2)_{m}(c+1/2)_{m}}
F\left(\begin{matrix}
2b, 2c, c+b \\
1/2+c+b+m, 2c+2b
\end{matrix}\,\bigg\vert\,R\,\bigg\vert\:x\right).
\end{equation}
Here
\begin{equation}\label{eq:ClasenRational}
R(k)=\frac{(\f+k)_{\m}(c+k/2)_{m}}{[(\f)_{\m}]^2(c+b+k/2)_{m}}\hat{R}(-b)
\end{equation}
is a rational function of $k$ with 
\begin{multline}\label{eq:hatRfinal}
\hat{R}(t)=\hat{T}_{m}(c,1/2-k/2;1-k/2)\hat{T}_{m}(1/2-k-m-c-b,f_r+k/2-M_{r-1};1-k/2)
(f_r+\cdot)_{m_r}
\\
\hat{T}_{M_{r-1}}(1-k-f_{r},f_{r-1}+k/2-M_{r-2};1-k/2)
(f_{r-1}+\cdot)_{m_{r-1}}
\\
\hat{T}_{M_{r-2}}(1-k-f_{r-1},f_{r-2}+k/2-M_{r-3};1-k/2)
\cdots
\hat{T}_{m_1}(1-k-f_2,f_1+k/2;1-k/2)(f_1+\cdot)_{m_1},
\end{multline}
where $M_s=m_1+\cdots+m_s$ for $s=1,\ldots,r$, $m=M_r$, and the second Miller-Paris operator $\hat{T}_{m}$ is defined by any of the equivalent formulas  \eqref{eq:hatQLagrange}, \eqref{eq:hatQMP}, \eqref{eq:hatQKP},  \eqref{eq:hatQKPNew} or \eqref{eq:hatQKPNewNew}; the variable $t$ is the argument of the image polynomial of the leftmost operator.
\end{theorem}

\remark We can rewrite the coefficient at  $\hat{R}(-b)$ in \eqref{eq:ClasenRational} in terms of Pochhammer's symbols indexed by $k$ as follows
\begin{multline*}
\frac{(\f+k)_{\m}(c+k/2)_{m}}{[(\f)_{\m}]^2(c+b+k/2)_{m}}=\frac{(c)_{m}(\f+\m)_{k}}{(\f)_{\m}(c+b)_{m}(\f)_{k}}
\\
\times\frac{(2c+1)_{k}(2c+3)_{k}\cdots(2c+2m-1)_{k}(2c+2b)_{k}(2c+2b+2)_{k}\cdots(2c+2b+2m-2)_{k}}{(2c)_{k}(2c+2)_{k}\cdots(2c+2m-2)_{k}(2c+2b+1)_{k}(2c+2b+3)_{k}\cdots(2c+2b+2m-1)_{k}}.
\end{multline*}

We will split the proof of Theorem~\ref{th:ClausenPerturbed} into three lemmas. The first lemma,  due to Kim, Rathie and Paris  \cite[Theorem~1]{KRP2013}, is immediate by comparing the power series coefficients on both sides of \eqref{eq:MP2general}. 
\begin{lemma}\label{lm:Saal_extended}
The following extension of Saalsch\"{u}tz's  identity holds
\begin{equation}\label{eq:Saal_extended}
{}_{r+3}F_{r+2}\left(\begin{array}{l}
-n, \alpha, \beta,  \f+\m \\
c, 1+\alpha+\beta-\gamma+m-n,  \f
\end{array}\right)=\frac{(\gamma-\alpha-m)_n(\gamma-\beta-m)_n(\llambda+1)_n}{(\gamma)_n(\gamma-\alpha-\beta-m)_n(\llambda)_n},
\end{equation}
 where $\llambda$ is the $m$-vector of the negated roots of the polynomial 
 $$
\hat{Q}_m(t)=\hat{Q}_m\!\left(\begin{matrix}\alpha,\beta\\\gamma\end{matrix}\:\bigg\vert \frac{(\f+\cdot)_{\m}}{(\f)_{\m}}\:\bigg\vert\: t\right)=\frac{1}{(\f)_{\m}}[\hat{T}_{m}(\alpha,\beta;\gamma)(\f+\cdot)_{\m}](t),
 $$
 i.e. the roots of the polynomial $\hat{Q}_m(-t)$.
\end{lemma}
The classical Whipple's theorem transforming very well-poised ${}_7F_6$ to Saalsch\"{u}tzian ${}_4F_3$ reads  \cite[Theorem~3.4.4]{AAR}:
\begin{multline}\label{eq:Whipple7F6-4F3}
{}_7F_6\left(\begin{matrix}
a, 1+a/2, b, c, d, e,-n \\
a/2, 1+a-b, 1+a-c, 1+a-d, 1+a-e, 1+a+n 
\end{matrix}\right)\\
\\
=\frac{(1+a)_{n}(1+a-c-e)_{n}}{(1+a-e)_{n}(1+a-c)_{n}} {}_4F_3\left(\begin{matrix}
1+a-b-d, c, e,-n \\
1+a-b, 1+a-d, c+e-a-n 
\end{matrix}\right).    
\end{multline}
In \cite[Theorem~3]{VF2022} Vyas and Fatawat found a polynomial perturbation of \eqref{eq:Whipple7F6-4F3} generalizing a previous result of Srivastava, Vyas and Fatawat \cite[Theorem~3.2]{SVF2019}. See also a related identity in Mishev \cite[Proposition~3.6]{Mishev2022}.  We record their formula in the  following lemma for which we provide a complete proof. This helps to keep the paper self-contained and, more importantly, to elucidate the structure of the underlying characteristic polynomial $\hat{W}$ in terms of the second Miller-Paris operator $\hat{T}$, which is not immediately apparent in \cite{VF2022}.
\begin{lemma}\label{lm:eq:genWhipple7F6}
The following generalized Whipple's transformation holds true:
\begin{multline}\label{eq:genWhipple7F6}
F\left(\begin{matrix}
a, 1+a/2, b, c, d, e, a-\f+1,   \f+\m, -n  \\
a/2, 1+a-b, 1+a-c, 1+a-d, 1+a-e, 1+a-\f-\m, \f,  1+a+n
\end{matrix}\right)
\\
=\frac{(1+a)_{n}(1+a-c-e)_{n}}{(1+a-c)_{n}(1+a-e)_{n}}
F\left(\begin{matrix}
c,e, 1+a-b-d-m, \etta+1, -n  \\
1+a-b, 1+a-d,c+e-a-n,\etta
\end{matrix}\right),
\end{multline}
where $\etta$ is the vector of negated roots of the polynomial defined by 
\begin{multline}\label{eq:hatW-defined}
\hat{W}_m(t)=\frac{1}{(\f)_{\m}}\hat{T}_{M_r}(d,f_r-b-M_{r-1};1+a-b)
(f_r+\cdot)_{m_r}
\\
\hat{T}_{M_{r-1}}(1+a-f_{r},f_{r-1}-b-M_{r-2};1+a-b)
(f_{r-1}+\cdot)_{m_{r-1}}
\\
\hat{T}_{M_{r-2}}(1+a-f_{r-1},f_{r-2}-b-M_{r-3};1+a-b)
\cdots
\hat{T}_{m_1}(1+a-f_2,f_1-b;1+a-b)(f_1+\cdot)_{m_1},
\end{multline}
where $M_{s}=m_1+m_2+\cdots+m_s$, $m=M_r$.
\end{lemma}

\begin{proof}
Taking $e=f_1+m_1$, $d=1+a-f_1$ in \eqref{eq:Whipple7F6-4F3},  the Saalsch\"{u}tzian ${}_4F_3$ function on the right hand side takes the form \eqref{eq:Saal_extended} with  by $r=1$, and hence, 
$$
{}_4F_3\left(\begin{matrix}
-n, f_1-b, c,f_1+m_1\\
1+a-b, c+f_1+m_1-a-n,f_1 
\end{matrix}\right)=\frac{(1+a-f_1-m_1)_n(1+a-b-c-m_1)_n(\etta_1+1)_n}{(1+a-b)_n(1+a-c- f_1-m_1)_n(\etta_1)_n},
$$
where $\etta_1$ are negated roots of 
$$
P_1(t)=\hat{Q}_{m_1}\!\left(\begin{matrix}f_1-b,c\\1+a-b\end{matrix}\:\bigg\vert \frac{(f_1+\cdot)_{m_1}}{(f_1)_{m_1}}\:\bigg\vert t\right)=\frac{1}{(f_1)_{m_1}}[\hat{T}(f_1-b,c;1+a-b)(f_1+\cdot)_{m_1}](t).
 $$
This leads to 
\begin{multline}\label{eq:7F6-one-m-shift}
{}_7F_6\left(\begin{matrix}
a, 1+a/2, b, c, a-f_1+1, f_1+m_1,-n  \\
a/2, 1+a-b, 1+a-c, 1+a-f_1-m_1, f_1, 1+a+n
\end{matrix}\right) \\
\quad=\frac{(1+a)_{n}(1+a-b-c-m_1)_{n}(\etta_1+1)_n}{(1+a-b)_{n}(1+a-c)_{n}(\etta_1)_n}.
\end{multline}
Next, following \cite{VF2022}, we will apply Bailey's transform given by:
\begin{equation}\label{eq:BaileyTrans}
\text{if}~\beta_k=\sum_{r=0}^{k}\alpha_{r}u_{k-r}v_{k+r},~~\gamma_k=\sum_{r=k}^{\infty}\delta_{r}u_{r-k}v_{r+k},~\text{then}~\sum_{k=0}^{\infty}\alpha_k\gamma_k=\sum_{k=0}^{\infty}\beta_k\delta_k,
\end{equation}
which can be verified by a rather straightforward rearrangement of  the order of summation.
Define
$$
\begin{gathered}
\alpha_r=\frac{\left(a, 1+a/2, b, c, 1+a-f_1, f_1+m_1 \right)_r(-1)^r}{\left(a/2, 1+a-b, 1+a-c, f_1,1+a-f_1-m_1\right)_r r!},
\\
u_r=\frac{1}{r!}, \quad v_r=\frac{1}{(1+a)_r} \quad \text { and } \quad \delta_r=\frac{(d, e,-n)_r}{(d+e-a-n)_r},
\end{gathered}
$$
where the shorthand notation $(a_1,a_2,\ldots,a_q)_r=(a_1)_{r}(a_2)_{r}\cdots(a_q)_{r}$ has been used. 
This yields
\begin{multline*}
\beta_k=
\sum_{r=0}^{k}\frac{\alpha_{r}(-k)_{r}(-1)^r}{k!(1+a)_{k}(1+a+k)_{r}}\\
=\frac{1}{(1+a)_{k}k!}{}_7F_6\left(\begin{matrix}
a, 1+a/2, b, c, 1+a-f_1,  f_1+m_1, -k  \\
a/2, 1+a-b, 1+a-c, f_1, 1+a-f_1-m_1, 1+a+k
\end{matrix}\right)
\\
=\frac{(1+a-b-c-m_1)_{k}(\etta_1+1)_k}{k!(1+a-b)_{k}(1+a-c)_{k}(\etta_1)_n}
\end{multline*}
by \eqref{eq:7F6-one-m-shift}.
Now, replacing $r$ by $r+k$, we get by an application of the Saalsch\"{u}tz theorem:
\begin{multline*}
\gamma_k=\sum_{r=0}^{\infty}\delta_{r+k}u_{r}v_{r+2k}
=\sum_{r=0}^{\infty}\frac{(d)_{r+k}(e)_{r+k}(-n)_{r+k}}{(d+e-a-n)_{r+k}r!(1+a)_{r+2k}}
\\
=\frac{(-n)_{k}(d)_{k}(e)_{k}}{(d+e-a-n)_{k}(1+a)_{2k}}\sum_{r=0}^{n-k}\frac{(d+k)_{r}(e+k)_{r}(-n+k)_{r}}{(1+a+2k)_{r}(d+e-a-n+k)_{r}r!}
\\
=\frac{(-n)_{k}(d)_{k}(e)_{k}}{(d+e-a-n)_{k}(1+a)_{2k}}
{}_3F_2\left(\begin{matrix}d+k,e+k,-n+k\\
1+a+2k,d+e-a-n+k\end{matrix}\right)
\\
=\frac{(-n)_{k}(d)_{k}(e)_{k}(1+a-d+k)_{n-k}(1+a-e+k)_{n-k}}{(1+a)_{2k}(d+e-a-n)_{k}(1+a+2k)_{n-k}(1+a-d-e)_{n-k}}
\\
=\frac{(-n)_{k}(d)_{k}(e)_{k}(1+a-d+k)_{n}(1+a-e+k)_{n}(1+a+k+n)_{k}(-1)^k}{(1+a)_{2k}(1+a+2k)_{n}(1+a-d-e)_{n}(1+a-d+n)_{k}(1+a-e+n)_{k}}
\\
=\frac{(-n)_{k}(d)_{k}(e)_{k}(1+a-d)_{n}(1+a-e)_{n}(1+a+k)_{k}(-1)^k}{(1+a)_{2k}(1+a+k)_{n}(1+a-d-e)_{n}(1+a-d)_{k}(1+a-e)_{k}}
\\
=\frac{(-1)^k(-n)_{k}(d)_{k}(e)_{k}(1+a-d)_{n}(1+a-e)_{n}}{(1+a)_{n}(1+a+n)_{k}(1+a-d-e)_{n}(1+a-d)_{k}(1+a-e)_{k}},
\end{multline*}
where we utilized the easily verifiable identities
$$
(A)_{n-k}=\frac{(A)_{n}(-1)^k}{(1-A-n)_{k}},~~~(A-k)_{k}=(-1)^k(1-A)_{k}
$$
and
$$
\frac{(1+a+k)_k}{(1+a)_{2k}(1+a+k)_{n}}=\frac{1}{(1+a)_{n}(1+a+n)_{k}}.
$$
Substituting these values of $\beta_k$, $\gamma_k$ into Bailey's formula \eqref{eq:BaileyTrans} leads to 
\begin{multline*}
\sum_{k=0}^{\infty}
\frac{\left(a, 1+a/2, b, c, d, e, 1+a-f_1, f_1+m_1,-n\right)_k}{\left(a/2, 1+a-b, 1+a-c, 1+a-d, 1+a-e, f_1,1+a-f_1-m_1,1+a+n\right)_k k!}
\\
=
{}_9F_8\left(\begin{matrix}
a, 1+a/2, b, c, d, e, 1+a-f_1, f_1+m_1,-n  \\
a/2, 1+a-b, 1+a-c, 1+a-d, 1+a-e, f_1, 1+a-f_1-m_1, 1+a+n
\end{matrix}\right)
\\
=\frac{(1+a-d-e)_{n}(1+a)_{n}}{(1+a-d)_{n}(1+a-e)_{n}}\sum_{k=0}^{n}\frac{(1+a-b-c-m_1)_{k}(\etta_1+1)_k}{k!(1+a-b)_{k}(1+a-c)_{k}(\etta_1)_n}\frac{(d, e,-n)_{k}}{(d+e-a-n)_{k}}
\\
=\frac{(1+a-d-e)_{n}(1+a)_{n}}{(1+a-d)_{n}(1+a-e)_{n}}
{}_{m_1+4}F_{m_1+3}\left(\begin{array}{l}
-n, d, e, 1+a-b-c-m_1,  \etta_1+1\\
1+a-b, 1+a-c, d+e-a-n, \etta_1
\end{array}\right).
\end{multline*}
The function on the right hand side is Saalsch\"utzian, and hence, this extends Whipple's theorem \eqref{eq:Whipple7F6-4F3}.  Then, we can repeat the same trick: taking $e=f_2+m_2$, $c=1+a-f_2$ in the above formula,  the Saalsch\"{u}tzian  ${}_{m_1+4}F_{m_1+3}$ function on the right hand side takes the form \eqref{eq:Saal_extended} and we obtain
\begin{multline*}
{}_{m_1+4}F_{m_1+3}\left(\begin{array}{l}
-n, d, f_2-b-m_1,  f_2+m_2, \etta_1+1\\
1+a-b, d+f_2+m_2-a-n, f_2, \etta_1
\end{array}\right)
\\
=\frac{(1+a-b-d-m_1-m_2)_n(1+a-f_2-m_2)_n(\etta_2+1)_n}{(1+a-b)_n(1+a-d-f_2-m_2)_n(\etta_2)_n},
\end{multline*}
where $\etta_2$ is the vector of negated roots of the polynomial 
\begin{multline*}
\hat{T}_{m_1+m_2}(d,f_2-b-m_1;1+a-b)\frac{(f_2+\cdot)_{m_2}(\etta_1+\cdot)_{1}}{(f_2)_{m_2}(\etta_1)_{1}}
\\
=\hat{T}_{m_1+m_2}(d,f_2-b-m_1;1+a-b)\hat{T}_{m_1}(1+a-f_2,f_1-b;1+a-b)\frac{(f_1+\cdot)_{m_1}}{(f_1)_{m_1}},    
\end{multline*}
and we substituted $c=1+a-f_2$ in the definition of the polynomial $P_1$ having roots $\etta_1$.  This leads to 
\begin{multline*}
{}_9F_8\left(\begin{matrix}
a, 1+a/2, b, d, a-f_1+1, 1+a-f_2,  f_1+m_1, f_2+m_2, -n  \\
a/2, 1+a-b, 1+a-d, 1+a-f_1-m_1, 1+a-f_2-m_2, f_1,  f_2,  1+a+n
\end{matrix}\right)
\\
=\frac{(1+a)_{n}(1+a-b-d-m_1-m_2)_{n}(\etta_2+1)_{n}}{(1+a-b)_{n}(1+a-d)_{n}(\etta_2)_n}.
\end{multline*}
We can now define new $\alpha_r$ as the summand of ${}_9F_8$ multiplied by $(-1)^r$ and with $(-n)_r/(1+a+n)_{r}$ omitted, and $u_r$, $v_r$, $\delta_r$ as before (we need to use new symbols in the definition of $\delta_r$ not appearing in $\alpha_r$, say $c$ and $e$) and apply Bailey's formula \eqref{eq:BaileyTrans} again.  Repeating these steps $r$ times we arrive at \eqref{eq:genWhipple7F6}.
\end{proof}

By stopping the algorithm in the proof of Lemma~\ref{lm:eq:genWhipple7F6} before the ultimate application of the Bailey transform, we obtain a polynomial perturbation of  Dougall’s terminating ${}_5F_4(1)$ summation theorem as given in \cite[Theorem~2]{VF2022}.
\begin{corollary}\label{cr:genDougall5F4}
The following generalized Dougall's ${}_5F_4$ summation holds true:
\begin{multline}\label{eq:genDougall5F4}
F\left(\begin{matrix}
a, 1+a/2, b, d, a-\f+1,   \f+\m, -n  \\
a/2, 1+a-b, 1+a-d, 1+a-\f-\m, \f,  1+a+n
\end{matrix}\right)
\\
=\frac{(1+a)_{n}(1+a-b-d-M_r)_{n}(\etta+1)_{n}}{(1+a-b)_{n}(1+a-d)_{n}(\etta)_n},
\end{multline}
where the vector $\etta$ comprises the roots of $\hat{W}$ defined by \eqref{eq:hatW-defined}.
\end{corollary}

Set now $1+a-b-d-m=c+e-a-n$ in \eqref{eq:genWhipple7F6} so that the hypergeometric function on the right hand side becomes 
$$
F\left(\begin{matrix}
c,e, \etta+1, -n  \\
1+a-b, 1+a-d,\etta
\end{matrix}\right),
$$
whereas the parametric excess equals`$1$, i.e. it is Saalsch\"{u}tzian. Hence, we are in the position to apply \eqref{eq:Saal_extended} to obtain \cite[Theorem~4]{VF2022}:
\begin{corollary}\label{cr:genDougall7F6}
Suppose $1+2a=c+e+b+d+m-n$, then 
\begin{multline}\label{eq:genDougall7F6}
F\left(\begin{matrix}
a, 1+a/2, b, c, d, e, a-\f+1,   \f+\m, -n  \\
a/2, 1+a-b, 1+a-c, 1+a-d, 1+a-e, 1+a-\f-\m, \f,  1+a+n
\end{matrix}\right)
\\
=\frac{(1+a)_{n}(1+a-c-e)_{n     }(1+a-b-c-m)_n(1+a-b-e-m)_n(\llambda+1)_n}{(1+a-c)_{n}(1+a-e)_{n}(1+a-b)_n(1+a-b-c-e-m)_n(\llambda)_n},
\end{multline}
where $\llambda$ is the $m$-vector of the negated roots of the polynomial 
 $$
\hat{Q}_m(t)=[\hat{T}_{m}(c,e;1+a-b)\hat{W}_{m}](t),
 $$
 where $\hat{W}_m$ is defined in  \eqref{eq:hatW-defined}. 
\end{corollary}
The above corollary is a generalization of Dougall’s summation theorem for  terminating very well-poised $2$-balanced ${}_7F_6(1)$, see \cite[Theorem~3.5.1]{AAR}.

Our final lemma will be derived from Corollary~\ref{cr:genDougall7F6}. It is the key ingredient of the proof of Theorem~\ref{th:ClausenPerturbed}. It gives the sum for polynomially perturbed well-poised Saalsch\"{u}tzian ${}_4F_3$. 

\begin{lemma}\label{lm:perturbedSaal4F3}
Suppose $k\in\N_0$, $m=m_1+\cdots+m_r$, $m_i\in\N$, and no denominator parameters is a non-positive integer. Then
\begin{multline}\label{eq:perturbedSaal4F3}
F\left(\begin{matrix}
-k, b, c, 1/2-k-c-b-m, 1-k-\f,   \f+\m  \\
1-k-b, 1-k-c, 1/2+c+b+m, 1-k-\f-\m, \f
\end{matrix}\right)
\\
=\frac{(1/2)_{m}(c+b+1/2)_{m}(2c)_{k}(2b)_{k}(c+b)_{k}}{(b+1/2)_{m}(c+1/2)_{m}(c)_{k}(b)_k(2c+2b)_{k}}\frac{(c+k/2)_{m}}{(c+b+k/2)_{m}(\f)_{\m}}\hat{R}(-b),
\end{multline}
where $\hat{R}$  is a rational function of $k$ given by \eqref{eq:hatRfinal}.
\end{lemma}
\begin{proof}
Set $b=a/2$ in \eqref{eq:genDougall7F6} to get 
\begin{multline*}
F\left(\begin{matrix}
a, c, d, e, a-\f+1,   \f+\m, -n  \\
1+a-c, 1+a-d, 1+a-e, 1+a-\f-\m, \f,  1+a+n
\end{matrix}\right)
\\
=\frac{(1+a)_{n}(1+a-c-e)_{n}(1+a/2-c-m)_n(1+a/2-e-m)_n}{(1+a-c)_{n}(1+a-e)_{n}(1+a/2)_n(1+a/2-c-e-m)_n}\hat{Q}_m(n),
\end{multline*}
where the balancing condition becomes $1+3a/2-d-m=c+e-n$ and 
\begin{multline*}
\hat{Q}_m(t)=\frac{1}{(\f)_{\m}}\hat{T}_{m}(c,e;1+a/2)\hat{T}_{M_r}(d,f_r-a/2-M_{r-1};1+a/2)
(f_r+\cdot)_{m_r}
\\
\hat{T}_{M_{r-1}}(1+a-f_{r},f_{r-1}-a/2-M_{r-2};1+a/2)
(f_{r-1}+\cdot)_{m_{r-1}}
\\
\hat{T}_{M_{r-2}}(1+a-f_{r-1},f_{r-2}-a/2-M_{r-3};1+a/2)
\cdots
\hat{T}_{m_1}(1+a-f_2,f_1-a/2;1+a/2)(f_1+\cdot)_{m_1}.
\end{multline*}
Next, set $e=(1+a)/2$ (so that $e=1+a-e$). The balancing condition now becomes  $d=1/2+a-m-c+n$ and the summation formula reduces to
\begin{multline*}
F\left(\begin{matrix}
a, c, d, a-\f+1,   \f+\m, -n  \\
1+a-c, 1+a-d, 1+a-\f-\m, \f,  1+a+n
\end{matrix}\right)
\\
=\frac{4^{n}(1/2+a/2-c)_{n}(1+a/2-c-m)_n(1/2-m)_n}{(1+a-c)_{n}(a+n+1)_{n}(1/2-c-m)_n}\hat{Q}_m(n),
\end{multline*}
where we used
$$
\frac{(a+1)_{n}}{(a/2+1/2)_{n}(a/2+1)_{n}}=\frac{4^n}{(a+n+1)_{n}}.
$$
We can rewrite the right hand side in terms of gamma functions as follows:
\begin{align*}
&\frac{4^{n}\Gamma(1+a-c)\Gamma(a+n+1)\Gamma(1/2+a/2-c+n)\Gamma(1/2-c-m)}{\Gamma(1/2+a/2-c)\Gamma(1+a/2-c-m)\Gamma(1/2-m)\Gamma(1+a-c+n)}
\\
&\times\frac{\Gamma(1+a/2-c-m+n)\Gamma(1/2-m+n)}{\Gamma(a+2n+1)\Gamma(1/2-c-m+n)}\hat{Q}_m(n).
\end{align*}
If we now replace $a$ by $-k\in\-\N_0$ and $-n$ by arbitrary $b$ in the above expression, we will get
\begin{align*}
&\frac{\Gamma(1-k-c)\Gamma(1-k-b)\Gamma(1/2-k/2-c-b)\Gamma(1/2-c-m)}{4^{b}\Gamma(1/2-k/2-c)\Gamma(1-k/2-c-m)\Gamma(1/2-m)\Gamma(1-k-c-b)}
\\
&\times\frac{\Gamma(1-k/2-c-b-m)\Gamma(1/2-b-m)}{\Gamma(1-k-2b)\Gamma(1/2-c-b-m)}\hat{Q}_m(-b).
\end{align*}
The balancing condition takes the form  $d=1/2-k-m-c-b$. Using the recurrence relation and Legendre's duplication formula for the gamma function we obtain
$$
\Gamma(1/2-A-k/2)\Gamma(1-A-k/2-m)=2^{2A+k}\sqrt{\pi}\frac{\Gamma(1-2A-k)}{(A+k/2)_{m}},
$$
so that the above expression becomes
$$
\frac{\Gamma(1/2-b-m)\Gamma(1/2-c-m)(c+k/2)_{m}\Gamma(1-k-c)\Gamma(1-k-b)\Gamma(1-k-2c-2b)\hat{Q}_m(-b)}{\Gamma(1/2-m)\Gamma(1/2-c-b-m)(c+b+k/2)_{m}\Gamma(1-k-2c)\Gamma(1-k-2b)\Gamma(1-k-c-b)}.
$$
Further, we have:
$$
\frac{\Gamma(1/2-b-m)\Gamma(1/2-c-m)}{\Gamma(1/2-m)\Gamma(1/2-c-b-m)}=
\frac{\sin(\pi/2)\sin(\pi(c+b+1/2))\Gamma(1/2+m)\Gamma(c+b+1/2+m)}{\sin(\pi(b+1/2))\sin(\pi(c+1/2))\Gamma(b+1/2+m)\Gamma(c+1/2+m)},
$$
and
$$
\frac{\Gamma(1-k-A)}{\Gamma(1-k-2A)}=2\cos(\pi{A})\frac{\Gamma(2A+k)}{\Gamma(A+k)}.
$$
Substituting these formulas, simplifying using the duplication formula  again and recalling that $d=1/2-k-m-c-b$ we finally obtain \eqref{eq:perturbedSaal4F3} as a formal identity. 
 To prove that this is indeed a true equality note that both sides are rational functions of $b$ and which coincide for infinite number of values $b=0,-1,-2\ldots$, implying that they must be identical. It remains to define $\hat{R}(t)=(\f)_{\m}\hat{Q}_{m}(t)$.
\end{proof}

\remark We can put $a=-k$, $-n=b$ in the first equality of the above proof and apply the same manipulations to the resulting expression as in the above proof. This leads to the following summation formula :
\begin{align*}
&F\left(\begin{matrix}
-k, b, c, e, 1-3k/2-c-e-b-m, 1-k-\f,  \f+\m  \\
1-k-b, 1-k-c, 1-k-e, k/2+c+e+b+m, 1-k-\f-\m, \f
\end{matrix}\right)=
\\
&\frac{(1/2)_{l}(c)_{m}(e)_{m}(b+c+e)_{m}(b)_{l}(b+c)_{k}(c+e)_{k}(b+e)_{k}(c+m)_{l}(e+m)_{l}(b+c+e+m)_{l}}{2^{-k}(b+c)_{m}(c+e)_{m}(b+e)_{m}(b)_{k}(c)_{k}(e)_{k}(b+c+e)_{k}(c+b+m)_{l}(c+e+m)_{l}(e+b+m)_{l}}\hat{Q}_m(-b),
\\&~~~~~~~\text{ if}~k=2l~\text{is even,}
\\
&0,~~~~\text{ if}~k=2l+1~\text{is odd.}
\end{align*}
Here $\hat{Q}_m$ is the same as given in the course of the proof with $d$ replaced by $1-3k/2-c-e-b-m$ in accordance with the balancing condition.  It is not obvious whether this formula leads to any product identity generalizing \eqref{eq:ClausenPerturbed}. 

\smallskip

We are now ready to complete \emph{the proof of Theorem~\ref{th:ClausenPerturbed}}. 
Using $(\alpha)_{k-j}=(-1)^{j}(\alpha)_{k}/(1-\alpha-k)_{j}$ by the Cauchy product we have
\begin{multline*}
\Bigg[F\left(\begin{matrix}
b, c, \f+\m  \\
1/2+c+b+m, \f
\end{matrix}\bigg\vert\:x\right)\bigg]^2
\\
=\sum_{k=0}^{\infty}x^k\sum_{j=0}^{k}\frac{(b)_{j}(c)_{j}(\f+\m)_{j}(b)_{k-j}(c)_{k-j}(\f+\m)_{k-j}}{(1/2+c+b+m)_{j}(\f)_{j}(1/2+c+b+m)_{k-j}(\f)_{k-j}j!(k-j)!}
\\
=\sum_{k=0}^{\infty}\frac{(b)_{k}(c)_{k}(\f+\m)_{k}x^k}{(1/2+c+b+m)_{k}(\f)_{k}k!}
F\left(\begin{matrix}
-k, b, c, 1/2-k-c-b-m, 1-k-\f,   \f+\m  \\
1-k-b, 1-k-c, 1/2+c+b+m, 1-k-\f-\m, \f
\end{matrix}\right)
\\
=\frac{(1/2)_{m}(c+b+1/2)_{m}}{(b+1/2)_{m}(c+1/2)_{m}}\sum_{k=0}^{\infty}\frac{(2c)_{k}(2b)_{k}(c+b)_{k}x^k}{(1/2+c+b+m)_{k}(2c+2b)_{k}k!}
\frac{(\f+\m)_{k}(c+k/2)_{m}}{(\f)_{k}(c+b+k/2)_{m}(\f)_{\m}}\hat{R}(-b)
\\
=\frac{(1/2)_{m}(c+b+1/2)_{m}}{(b+1/2)_{m}(c+1/2)_{m}}
F\left(\begin{matrix}
2b, 2c, c+b \\
1/2+c+b+m, 2c+2b
\end{matrix}\,\bigg\vert\,R\,\bigg\vert\:x\right),
\end{multline*}
where we applied \eqref{eq:perturbedSaal4F3} and
$R(k)$ as given by \eqref{eq:ClasenRational}
is a rational function of $k$ since the operator $\hat{T}_{M_j}$ preserves rationality.  $\hfill\square$

\begin{corollary}\label{cr:oneunitshift}
For the values of parameters such that both sides make sense, the following product formula holds
\begin{equation}\label{eq:ClausenPerturbed1}
\Bigg[{}_{3}F_{2}\!\left(\begin{matrix}c,b,f+1\\c+b+3/2,f\end{matrix}\:\bigg\vert\,x\right)\Bigg]^2={}_{6}F_{5}\!\left(\begin{matrix}2c,2b,c+b,\lambda_1+1,\lambda_2+1,\lambda_3+1\\c+b+3/2,2c+2b+2,\lambda_1,\lambda_2,\lambda_3\end{matrix}\:\bigg\vert\,x\right),
\end{equation}
where $\lambda_1,\lambda_2,\lambda_3$ are the roots of the characteristic polynomial
\begin{equation}\label{eq:CharacterClausen1}
P_3(t)=f(f-t)(t^2-(8cb+4c+4b+1)t+2(2c+1)(2b+1)(c+b))+2cbt(t+1)(c+b-t+1/2) 
\end{equation}
of degree $3$.
\end{corollary}

\remark  Denoting $\tilde{P}_3(t):=P_3(-t)$ we can also write the right-hand side of the above identity as
$$
\frac{1}{\tilde{P}_3(0)}\!F\!\left(\begin{matrix}2c,2b,c+b\\c+b+3/2,2c+2b+2\end{matrix}\:\bigg\vert \tilde{P}_3\:\bigg\vert x\right),~\text{where}~\tilde{P}_3(0)=2f^2(2c+1)(2b+1)(c+b).
$$
\begin{proof}
We have $m_1=m=1$, $r=1$, $f_1=f$, so that by \eqref{eq:ClasenRational}, \eqref{eq:hatRfinal} 
$$
R(k)=\frac{(f+k)(2c+k)}{f^2(2c+2b+k)}\hat{R}(-b),
$$
where 
$$
\hat{R}(-b)=[\hat{T}_1(c,1/2-k/2;1-k/2)\hat{T}_1(-1/2-k-c-b,f+k/2;1-k/2)(f+\cdot)](-b).
$$
Using any of the formulas \eqref{eq:hatQLagrange}, \eqref{eq:hatQMP}, \eqref{eq:hatQKP}, \eqref{eq:hatQKPNew} and simplifying we get 
$$
R(k)\!=\!\frac{f(f+k)(k^2+(8cb+4c+4b+1)k+2(2c+1)(2b+1)(c+b))\!+\!2cbk(k\!-\!1)(c\!+\!b\!+\!k\!+\!1/2)}{f^2(2b+2c+k)(2b+2c+k+1)}.
$$
Substituting we arrive at the claim.
\end{proof}

In \cite[Lemma~2.1]{KPITSF2017} we demonstrated how product formulas lead to evaluation of certain Laplace convolution integrals. In the following corollary we illustrate it on the product formula  \eqref{eq:ClausenPerturbed1} extending \cite[Example~2.3]{KPITSF2017}. Certainly, a similar evaluation can be obtained for the general product formula \eqref{eq:ClausenPerturbed}.
\begin{corollary}
Write $\llambda=(\lambda_1,\lambda_2,\lambda_3)$ for the roots of $P_3$ defined in \eqref{eq:CharacterClausen1}. Then 
\begin{align*}
\frac{cb(f+1)}{(2c+2b+3)f}\int_{0}^{t}{}_{3}F_{3}\!&\left(\begin{matrix}c+1,b+1,f+2\\2,c+b+5/2,f+1\end{matrix}\:\bigg\vert\,t-u\right)\!
{}_{3}F_{3}\!\left(\begin{matrix}c+1,b+1,f+2\\2,c+b+5/2,f+1\end{matrix}\:\bigg\vert\,u\right)\!du
\\
=\frac{f(c+b)(\llambda+1)_1}{(f+1)(c+b+1)(\llambda)_1}&{}_{6}F_{6}\!\left(\begin{matrix}2c+1,2b+1,c+b+1,\lambda_1+2,\lambda_2+2,\lambda_3+2\\2,c+b+5/2,2c+2b+3,\lambda_1+1,\lambda_2+1,\lambda_3+1\end{matrix}\:\bigg\vert\,t\right)
\\
&-{}_{3}F_{3}\!\left(\begin{matrix}c+1,b+1,f+2\\2,c+b+5/2,f+1\end{matrix}\:\bigg\vert\,t\right).
\end{align*}
\end{corollary}

\begin{corollary}\label{cr:onedoubleshift}
For the values of parameters such that both sides make sense, the following product formula holds
\begin{equation}\label{eq:ClausenPerturbed2}
\Bigg[{}_{3}F_{2}\!\left(\begin{matrix}c,b,f+2\\c+b+5/2,f\end{matrix}\:\bigg\vert\,x\right)\Bigg]^2={}_{9}F_{8}\!\left(\begin{matrix}2c,2b,c+b,\llambda+1\\c+b+5/2,2c+2b+4,\llambda\end{matrix}\:\bigg\vert\,x\right),
\end{equation}
where $\llambda$ is the vector comprised of the roots of the $6$-th degree characteristic polynomial
\begin{equation}\label{eq:CharacterClausen2}
P_6(t)=A_0-A_1t+A_2t^2-A_3t^3+A_4t^4-A_5t^5+A_6t^6
\end{equation}
with coefficients given by \emph{(}where $s=b+c$, $p=bc$, $F=f(f+1)$ for brevity\emph{)} 
$$
A_0=64f^2(f+1)^2(c+1/2)_{2}(b+1/2)_{2}(c+b)_{2},
$$
\vspace*{-0.35in}
\begin{multline*}
A_1=p^2\left[8\left(4F(4f+1)-3\right)s^2+16\left(F\left(8f^2+16f+1\right)-3\right)s+2 \left(32F^2-9\right)\right]
\\
+p\big[8\left(2F(16f+5)-3\right)s^3+8s^2\left(2F\left(16 f^2+52 f+9\right)-9\right)+2s\left(2F\left(108f^2+188 f+7\right)-33\right)
\\
+2\left(2F\left(40 f^2+40 f-9\right)-9\right)\big]
+
48f\left(2f^2+3f+1\right)s^4
\\
+48F\left(2f^2+8f+3\right)s^3
+12F\left(18f^2+40f+11\right)s^2+12F\left(11f^2+17f+3\right)s+18F^2,
\end{multline*}
\vspace*{-0.35in}
\begin{multline*}
A_2=
p^2[4s^2\left(24f^2+16f+3\right)+ 
    8s(2f(16f^2+27f+8)-1)+16F\left(4f^2+12f+1\right)-15]
 \\
  + p[4s^3(32f+44f^2-1) + 4s^2\left(128 f^3+276 f^2+112 f-15\right)
 \\   
 + s(12f\left(12f^3+96 f^2+128 f+33\right)-107)  
 +2f\left(72f^3+304f^2+279f+29\right)-51]
  \\  
+3F\big[16s^4+16s^3(4f+5)+4s^2\left(6f^2+42f+29\right)+4s\left(9f^2+31f+14\right) +\left(11 f^2+23 f+6\right)\big],
\end{multline*}
\vspace*{-0.35in}
\begin{multline*}
A_3=
2p^2\left[4s^2(4f+1)+4s(4f+3)(6f+1)+8f\left(8 f^2+17 f+6\right)-1\right]
\\
+2p\big[12s^3(2f+1)+8s^2\left(22f^2+25f+4\right)+s\left(8f\left(18f^2+63f+41\right)+11\right)+8f^4+160f^3
\\
+326f^2+141f-10\big]
+96Fs^3+144s^2F(2+f)+24sF\left(f^2+10 f+10\right)+3F\left(6f^2+28f+17\right),
\end{multline*}
\vspace*{-0.25in}
\begin{multline*}
A_4=
p^2\left(96 f^2+8 (8 f+3)s+112 f+4s^2+15\right)
\\
+p\left(32 f^3+\left(196 f^2+340 f+103\right) s+246 f^2+(96 f+60)s^2+258 f+4s^3+43\right)
\\
+3f\left(f^3+14 f^2+30 f+17\right)+72 f (f+1),
\end{multline*}
\vspace*{-0.25in}
\begin{multline*}
A_5\!=\!8p^2(s+4f+2)+p\left((52f+44)s+20f^2+74f+8s^2+38\right)+24f(f+1)s+3f\left(2 f^2+9 f+7\right),
\end{multline*}
\begin{equation*}
A_6=4p^2 + (4(2+f)+4s)p +3f(1 + f).
\end{equation*}
\end{corollary}

\remark  Denoting $\tilde{P}_6(t):=P_6(-t)$ we can also write the right-hand side of the above identity as
$$
\frac{1}{\tilde{P}_6(0)}\!F\!\left(\begin{matrix}2c,2b,c+b\\c+b+5/2,2c+2b+4\end{matrix}\:\bigg\vert\, \tilde{P}_6\,\bigg\vert x\right),~\text{where}~\tilde{P}_6(0)=A_0.
$$
\begin{proof}
We have $m_1=m=2$, $r=1$, $f_1=f$, so that  by \eqref{eq:ClasenRational}, \eqref{eq:hatRfinal} $$
R(k)=\frac{(f+k)_{2}(c+k/2)_{2}}{[(f)_{2}]^2(c+b+k/2)_{2}}\hat{R}(-b)
$$
with 
$$
\hat{R}(-b)=[\hat{T}_{2}(c,1/2-k/2;1-k/2)\hat{T}_{2}(-3/2-k-c-b,f+k/2;1-k/2)
(f+\cdot)_{2}](-b).
$$
Using any of the formulas \eqref{eq:hatQLagrange}, \eqref{eq:hatQMP}, \eqref{eq:hatQKP}, \eqref{eq:hatQKPNew} and simplifying with the help of \textit{Wolfram Mathematica}
 we get \eqref{eq:CharacterClausen2}.
\end{proof}

\remark For the case $r=2$, $m_1=m_2=1$, $m=2$, we have 
$$
R(k)=\frac{(f_1+k)(f_2+k)(2c+k)(2c+k+2)}{[f_1f_2]^2(2c+2b+k)(2c+2b+k+2)}\hat{R}(-b)
$$
with
\begin{multline*}
\hat{R}(-b)=\big[\hat{T}_{2}(c,1/2-k/2;1-k/2)\hat{T}_{2}(-3/2-k-c-b,f_r+k/2-1;1-k/2)
(f_2+\cdot)
\\
\hat{T}_{1}(1-k-f_2,f_1+k/2;1-k/2)
(f_{1}+\cdot)\big](-b).
\end{multline*}
Using \textit{Wolfram Mathematica} we can show that again we can absorb denominator into the parameters of the hypergeometric series leaving a characteristic polynomial of degree $8$.  It is, however, too large to be cited here. 

Several other applications of product identities can be found in the literature. Let us mention two.  In \cite{Grinshpan} the author applies them to derive weighted $L_p$ norm inequalities for the generalized hypergeometric functions. In \cite[Corollary~3.5]{MFMP2024} the authors show that each product identity corresponds an expression for the finite additive free convolution (a.k.a. Walsh composition) of hypergeometric polynomials  in terms of another  hypergeometric polynomial.

\section{Concluding remarks}
The Euler-Pfaff transformations \cite[Theorem~2.2.5]{AAR} whose polynomial perturbations are known as Miller-Paris transformations can be viewed as product identities with one factor of the form ${}_1F_0(\alpha;-;x)=(1-x)^{-\alpha}$. Theorem~\ref{th:ClausenPerturbed} shows that other product identities with both hypergeometric factors having the orders higher than ${}_1F_0$ may also possess polynomial perturbations.  This demonstration is the main goal of this paper. At the same time,  Theorem~\ref{th:ClausenPerturbed} leaves many questions unanswered. Here are some: 

\begin{enumerate}
    \item Computation of the characteristic rational function $R$ by \eqref{eq:ClasenRational}, \eqref{eq:hatRfinal} requires the knowledge of the numbers $f_j$ which are negated roots of the perturbing polynomial $F_m(t)=(\f+t)_{\m}/(\f)_{m}$. However, computation of the left hand side of \eqref{eq:ClausenPerturbed} only requires the knowledge of the values $F_m(k)$ at nonnegative integers. So it is natural to ask whether one can find a formula for the rational function $R$ which only uses the values  $F_m(k)$, similarly to \eqref{eq:hatQKPNewNew}.

    \item Examples show that the denominator of the rational function $R$ can be absorbed in the parameters of the hypergeometric function on the right hand side, and we are left with the numerator polynomial. How to prove this in general? What is the degree of this polynomial? Once these questions have been answered, one can apply Lemma~\ref{lm:interpolation} to build such characteristic polynomial in a much simpler form. 

    \item Corollary~\ref{cr:onedoubleshift} shows that the characteristic polynomial quickly becomes hardly manageable even for the quadratic perturbation. Is there a simpler way of writing the characteristic rational function $R$ or its numerator polynomial?  More generally, is it possible to change the explicit parameters of the hypergeometric series on the right hand side of \eqref{eq:ClausenPerturbed} so that $R$ takes a simpler form?
\end{enumerate}

Finally, it seems very natural to assume that there should exist polynomial perturbations of various other product identities known in the literature beyond Clausen's formula \eqref{eq:Clausen}. Finding them is an obvious topic for future research.


\begin{thebibliography}{99}

\bibitem{ASZ2011} G.\:Almkvist, D.\:van\:Straten and W.\:Zudilin
Generalizations of Clausen's Formula and algebraic transformations of Calabi–Yau differential equations,
Proceedings of the Edinburgh Mathematical Society 54 (2011), 273--295.
\url{https://dx.doi.org/10.1017/S0013091509000959}

		
\bibitem{AAR} G.E.\:Andrews, R.\:Askey, R.\:Roy, Special Functions, Cambridge University Press, 1999. \url{https://doi.org/10.1017/CBO9781107325937}

\bibitem{Askey1989}R.\:Askey, Variants of Clausen’s formula for the square of a special ${}_2F_1$, in: Number Theory and Related Topic, in: Tata Inst. Fund. Res. Stud. Math., vol. 12, Oxford University Press, Oxford, 1989,  1--14. 
\url{https://www.math.tifr.res.in/~publ/studies/Number-Theory-And-Related-Topics.pdf}

\bibitem{CTYZ2011} H.H.\:Chan, Y.\:Tanigawa, Y.\:Yang, W.\:Zudilin,
New analogues of Clausen’s identities arising from the theory of modular forms, Advances in Mathematics 228 (2011) 1294--1314.
\url{http://dx.doi.org/10.1016/j.aim.2011.06.011}


\bibitem{Grinshpan}A.Z.\:Grinshpan, Generalized hypergeometric functions: product identities and weighted norm inequalities
Ramanujan Journal 31 (2013), 53--66.
\url{https://doi.org/10.1007/s11139-013-9487-x}


\bibitem{Karlsson} P.W.\:Karlsson, Hypergeometric functions with integral parameter differences, J Math Phys.
1971;12:270–271. \url{https://doi.org/10.1063/1.1665587}
		
\bibitem{KPITSF2017} D.\:Karp and E.\:Prilepkina,  Applications of the Stieltjes and Laplace transform representations of the hypergeometric functions, Integral Transforms and Special Functions, volume 28, no.10 (2017), 710--731.
\url{http://dx.doi.org/10.1080/10652469.2017.1351964}
				
\bibitem{KPITSF2018}D.B.\:Karp and E.G.\:Prilepkina, Extensions of Karlsson–Minton summation theorem and some consequences of the first Miller–Paris transformation, Integral Transforms and Special Functions, Vol. 29, Issue 12 (2018), 955--970.
\url{https://doi.org/10.1080/10652469.2018.1526793}
		
\bibitem{KPResults2019}D.B.\:Karp and E.G.\:Prilepkina, Degenerate Miller–Paris transformations, Results Math (2019) 74:94.
\url{https://doi.org/10.1007/s00025-019-1017-8}
		
\bibitem{KPChapter2020} D.B.\:Karp and E.G.\:Prilepkina, Alternative approach to Miller-Paris transformations and their extensions, pp.117-140 in Transmutation Operators and Applications (edited by V.V.Kravchenko and S.M.Sitnik), Springer Trends in Mathematics Series, Birkh\"{a}user, 2020. \url{https://doi.org/10.1007/978-3-030-35914-0_6}

\bibitem{KPSymmetry2022}D.B.\:Karp and E.G.\:Prilepkina, Beyond the beta integral method: transformation formulas for hypergeometric functions via Meijer’s G function, Symmetry 2022, 14(8), 1541.  \url{https://doi.org/10.3390/sym14081541}

\bibitem{KRP2013} Y.S.\:Kim, A.K.\:Rathie, and R.B.\:Paris, An extension of Saalsch\"{u}tz’s summation theorems for the series ${}_{r+3}F_{r+2}(1)$, Integral Transforms Spec. Funct., vol. 24, no. 11, 2013, pp. 916--921. \url{http://dx.doi.org/10.1080/10652469.2013.777721}

\bibitem{LukeBookVol1} Y.L.\:Luke, The special functions and their Approximations, Volume I, Academic Press, Inc., 1969. \url{https://ia801506.us.archive.org/10/items/in.ernet.dli.2015.141299/2015.141299.The-Special-Functions-And-Their-Approximations-Vol-1.pdf}

\bibitem{Maier2019}R.S.\:Maier, 
Extensions of the classical transformations of the hypergeometric function, Advances in Applied Mathematics,  Volume 105,  2019, 25--47. \url{https://doi.org/10.1016/j.aam.2019.01.002}

\bibitem{MFMP2024}A.\:Martínez-Finkelshtein, R.\:Morales, D.\:Perales,
Real Roots of Hypergeometric Polynomials via Finite Free Convolution, International Mathematics Research Notices, Volume 2024, Issue 16, 2024,  11642--11687, \url{https://doi.org/10.1093/imrn/rnae120}
	
\bibitem{MP2012}A.R.\:Miller and R.B.\:Paris, 
On a result related to transformations and summations of	generalized hypergeometric series, Math. Commun. 17(2012), 205--210. \url{https://hrcak.srce.hr/clanak/123532}
		
\bibitem{MP2013}A.R.\:Miller and R.B.\:Paris, Transformation formulas for the generalized hypergeometric function with integral parameter differences, Rocky Mountain Journal Of Mathematics Volume 43, Number 1 (2013), 291--327. \url{https://doi.org/10.1216/RMJ-2013-43-1-291}

\bibitem{Minton} B.M.\:Minton, Generalized hypergeometric functions at unit argument, J Math Phys, 12(1970), 1375--1376.
\url{https://doi.org/10.1063/1.1665270}

\bibitem{Mishev2022}I.D.\:Mishev, Extensions of classical hypergeometric identities of Bailey and Whipple, J. Math.Anal.Appl.507(2022) 125775.
\url{https://doi.org/10.1016/j.jmaa.2021.125775}
		
\bibitem{Norlund} N.E.\:N{\o}rlund, Hypergeometric functions, Acta Math. 94 (1955), 289--349. \url{https://doi.org/10.1007/BF02392494}
		
	
\bibitem{RJRJR1992} K.\:Srinivasa Rao, J.\: Van der Jeugt, J.\:Raynal, R.\:Jagannathan and V.\: Rajeswari, Group theoretical basis for the terminating ${}_3F_2(1)$ series, J. Phys. A: Math. Gen. 25(1992), 861. \url{https://doi.org/10.1088/0305-4470/25/4/023}
		
\bibitem{RDN2002}  K.\:Srinivasa\:Rao, H.D.\:Doebner,  P.\:Natterman,  Generalized hypergeometric series and the symmetries of $3-j$ and $6-j$ coefficients. In Number Theoretic Methods. Developments in Mathematics; Kanemitsu S., Jia C. Eds.; Springer: Boston, MA, USA, 2002; Volume 8, pp. 381--403. \url{http://dx.doi.org/10.1007/978-1-4757-3675-5_20}

\bibitem{Schlosser} M.J.\:Schlosser, Multilateral transformations of $q$-series with quotients of parameters that are nonnegative integral powers of $q$, in: $q$-Series with Applications to Combinatorics, Number Theory, and Physics (B. C. Berndt, K. Ono, eds.), Contemporary Mathematics 291, AMS, Providence, RI, 2001, 203--227. \url{https://bookstore.ams.org/conm-291/}
		
\bibitem{SVF2019} H.M.\: Srivastava, Y.\:Vyas, K.\:Fatawat, Extensions of the classical theorems for very well-poised hypergeometric functions, Rev. R. Acad. Cienc. Exactas Fís. Nat., Ser. A Mat. 113(2) (2019) 367--397. \url{https://doi.org/10.1007/s13398-017-0485-5}

\bibitem{VF2022} Y.\:Vyas, K.\:Fatawat, Summations and transformations for very well-poised hypergeometric functions ${}_{2q+5}F_{2q+4}(1)$ and ${}_{2q+7}F_{2q+6}(1)$ with arbitrary integral parameter differences, Miskolc Mathematical Notes, Vol. 23 (2022), No. 2, pp. 957--973. \url{https://doi.org/10.18514/MMN.2022.3427}

\bibitem{Vidunas2011} R.\:Vidunas, A generalization of Clausen’s identity, Ramanujan Journal 26 (2011),  133--146. \url{https://doi.org/10.1007/s11139-011-9314-1}


\bibitem{WangRathie2013} X.\:Wang and A.K.\:Rathie,  Extension of quadratic transformation due to Whipple with an application. Adv. Differ. Equations 2013, 157. \url{http://dx.doi.org/10.1186/1687-1847-2013-157}

\end{thebibliography}
\end{document}